\documentclass{article}

\usepackage[utf8]{inputenc}
\usepackage{subfiles}\usepackage{blindtext}

\usepackage{graphicx,mathrsfs,tikz,tikz-cd,latexsym,ifthen,amsmath, amssymb,  amsfonts,amssymb,amsthm,stmaryrd,fancyhdr,amscd,amsbsy,amstext,tensor, amsfonts, enumerate,enumitem,empheq,mathtools,mathpazo,verbatim,a4wide, epigraph, physics, imakeidx, pdfpages}
\usetikzlibrary{positioning}
\usetikzlibrary{calc}

\mathtoolsset{showonlyrefs=true}
\newcommand{\calU}{\mathcal{U}}

\newcommand{\tu}{{\calU}}
\newcommand{\tuz}{\tu_\Z}
\newcommand{\uno}{{1\!\!\!\!1}}

\newcommand{\const}{M}
\newcommand{\R}{{\mathbb{R}}}
\newcommand{\N}{{\mathbb{N}}}

\newcommand{\Z}{{\mathbb{Z}}}
\newcommand{\Q}{{\mathbb{Q}}}

\newcommand{\uz}{\calU_\Z}

\renewcommand{\u}{\calU}

\newcommand {\buz}{{\bar{\mathcal{U}}}_\Z}

\newcommand {\muz}{{{\check{\mathcal{U}}}}_{\Z}}

\newcommand {\bu}{\bar{{\mathcal{U}}}}

\newcommand {\h}{\mathfrak h}

\newtheorem{theorem}{Theorem}

\newtheorem{remark}[theorem]{Remark}
\newtheorem{notation}[theorem]{Notation}
\newtheorem{definition}[theorem]{Definition}
\newtheorem{proposition}[theorem]{Proposition}
\newtheorem{lemma}[theorem]{Lemma}
\newtheorem{corollary}[theorem]{Corollary}

\newtheorem*{corollary*}{Corollary}
\newtheorem*{theorem*}{Theorem}
\newtheorem*{proposition*}{Proposition}
\newtheorem*{conjecture*}{Conjecture}
\newtheorem*{lemma*}{Lemma}

\numberwithin{equation}{theorem}

\usepackage{enumitem}
\newlist{abbrv}{itemize}{1}
\setlist[abbrv,1]{label=,labelwidth=1in,align=parleft,itemsep=0.1\baselineskip,leftmargin=!}

\title{Divided Power Integral forms of Affine Algebras}
\author{Margherita Paolini}
\date{July 2024}

\begin{document}

\maketitle

\begin{abstract}
In this paper we shall prove that the $\Z$-subalgebra generated by the
divided powers of the Drinfeld generators $x^{\pm}_{i,r}\;(i,r \subseteq I\times \Z) $ of an affine KacMoody algebra is an integral form.
We compare this integral form with the analogous by mean the Chevalley generators studied by Mitzman’s and Garland's (see [M,G]).
We shall prove that the integral forms coincide outside type $A_{2n}^{(2)}$ and it is strictly smaller in the latter case. Moreover if $n>1$  a completely new fact emerge, that is  the subalgebra generated by the imaginary vectors is in fact not a polynomial algebra and we describe it's structure.
Moreover to get around this problem, we introduces two other integral forms in $A_2^{(2)}$, in order to obtain the desirable polynomial property.

In this paper we shall prove that the integral subalgebra generated by the
divided powers of the Drinfeld generators of an affine KacMoody algebra is an integral form.
We compare this integral form with the analogous by mean the Chevalley generators studied by Mitzman’s and Garland's.
We shall prove that the integral forms coincide outside the twisted A type and it is strictly smaller in the latter case. Moreover if the rank of the algebra is greater than one a completely new fact emerge, that is  the subalgebra generated by the imaginary vectors is in fact not a polynomial algebra and we describe it's structure.
Moreover to get around this problem, we introduces two other integral forms in low rank case, in order to obtain the desirable polynomial property.
\end{abstract}
\tableofcontents

\section{Introduction}
Let $X^{(k)}_{\tilde n}$ be  an affine Kac-Moody algebra (see Section \ref{AffineKM}) and $\u$ its universal enveloping algebra.
The aim of this work is to give a basis over $\Z$ of the $\Z$-subalgebra of $\u$  generated by the divided powers of the Drinfeld generators that we will denote by $\uz$ (see Definition \ref{KacMoody}).
The integral forms for finite dimensional semisimple Lie algebras were first introduced by Chevalley in \cite{Ch} for the study of the Chevalley groups and of their representation theory.
Kostant (see \cite{Ko}) constructed the “divided power”-$\Z$-form for universal enveloping algebra $\u$ of the simple finite dimensional Lie algebras $X_n$, namely the $\Z$ subalgebra of $\u$ generated by the divided powers of the Chevalley generators $\{e_{i},f_{i}\mid i=\{1,\dots,n\}\} $. 
This construction it has been generalized to the untwisted
affine Kac-Moody algebras by Garland in \cite{HG}, the same result has been proved for all the twisted affine Kac-Moody algebras by Mitzman in \cite{DM} (see Theorem \ref{KMintegral}),
the two authors study the $\Z$-subalagebra, denoted by $\uz^{K.M.}$, generated by the divided powers of the Kac Moody generators (see Definition). It is well known that any affine Kac-Moody algebras admits another presentation, that is, the loop presentation.
Comparing the Kac-Moody presentation of the affine Kac-Moody algebras with its presentation as current algebra, one can notice a difference between the case of $A_{2n}^{(2)}$ and the other cases, which is at the origin of our work.
In \cite{DP} we studied the $\Z$-subalgebra $\uz$ of $\u$ generated by the divided powers of the Drinfeld generators $(x_{i,r}^{\pm})^{(k)}$ in the case when $X^{(k)}_n=A_{1}^{(1)}$ and $X^{(k)}_n=A_{2}^{(2)}$, namely the affine algebras of rank equal to one.
In the present work we prove that the $\Z$-subalgebra generated by $$\{(x_{i,r}^+)^{(k)},(x_{i,r}^-)^{(k)}\mid r\in\Z,k\in\N, i\in I\}$$ is an integral form of the enveloping algebra, we exhibit a basis generalizing the one provided in \cite{HG} and in \cite{DM}.
In fact we show that $\uz^{K.M.}=\uz$  if and only if  $X_{\tilde n}^{(k)}\neq A_{2n}^{(2)}$. 
In the case $X_{\tilde n}^{(k)}= A_{2n}^{(2)}$ it is always true that $\uz\subseteq\uz^{K.M.}$, in general we get two different $\Z$-subalgebras of $\u$, more precisely $\uz\subsetneq\uz^{K.M.}$, that is when there exists a vertex $i$ 
whose corresponding rank 1 subalgebra is not a copy of $A_{1}^{(1)}$ but is a copy of $A_{2}^{(2)}$ (see Theorem \ref{KMintegral} and Remark \ref{embeddings}).
Thus in order to complete the description of $\uz$ we need to study the case of $A_{2n}^{(2)}$ for $n>1$.
\\
The main differences that emerge in the latter case are as follows:\\
1) the positive real roots part, $\uz^{re,+}$ =($\uz\cap \u^{re,+}$), is not longer generated by the divided powers of positive real roots vectors but it is strictly contained,\\
2) studying the rank 1 cases we prove in \cite{DP} that, both in the untwisted and in the twisted case, $\uz^{im,+}$ =($\uz\cap \u^{im,+}$), where $\u^{im,+}$ is the subalgebra of $\u$ generated by the positive imaginary root vectors, is an algebras of polynomials. In the higher rank the situation changes:  it is no longer true that  $\u_{\Z}^{im,+}$ is an algebra of polynomial if $n>1$ (see Proposition \ref{notpol}), this is the reason why we decided to introduce another integral form $\muz$ adding extra elements (see Definition \ref{bhuz} and Theorem \ref{muztheorem}), in order to have the desirable polynomial property. 

The paper is organized as follows:

In Section \ref{Commutative} we introduce different $\Z$-subalgebras of the commutative polynomial algebra  $\Q[h_r\mid r\in \Z_{>0}]$, that is
$\Z^{(sym)}[h_r\mid r\in \Z_{>0}]$, 
$\Z_\lambda[\hat h_r\mid r\in \Z_{>0}]$ and $\Z^{(mix)}[h_r\mid r\in \Z_{>0}]$ (see Definitions \ref{hcappucciof} and \ref{mixdef} and Theorem \ref{bun}).
The first introduced in \cite{DP}, the second introduced by Garland \cite{HG}, both of which have already been studied and proved to be two isomorphic integral forms (see \cite{DP}, Appendix B): we recall such results in Subsection \ref{symmetric} (see Theorem \ref{bun}). The last one emerges entirely new in the study of $A_{2n}^{(2)}$, when $A_{2}^{(2)}$ is seen has copy of the first node of the Dynkin Diagram of $A_{2n}^{(2)}$ (see Section \ref{theothers}).
Subsection \ref{mixedsection} is dedicated to the study of $\Z^{(mix)}[h_r\mid r\in \Z_{>0}]$, we proved that event thought is not an algebra of polynomials (see Proposition \ref{notpol}), it is an integral form providing two kind of basis (see Definition \ref{lambdabase} and Theorems \ref{lambdabaseth} and \ref{polbasis}).\\
Section \ref{AffineKM}
is devoted to introduce notations and recollect previous results on integral forms on Affine Kac-Moody algebras. In Subsection \ref{setup} we recall the Kac-Moody presentation and the loop presentation (see Definitions \ref{KacMoody} and \ref{KacMoody2}) and the isomorphism that connect them (see Remark \ref{isomorphism}). 
In Subsection \ref{MitzmaneGarland} we recall the results of Garland and Mitzman about integral forms, then we describe the connection between the $\uz^{K.M.}$ and $\uz$ outside the $A_{2n}^{(2)}$.\\
The other section addressed the case of $A_{2n}^{(2)}$.\\
In Section \ref{theothers} we present two other integral forms of $A_{2}^{(2)}$ that we denote by $\buz$ and $\muz$ in order to distinguish to $\tuz$. 
$\buz$ is generated by the divided powers of the Drinfeld generators $x^\pm_{r}$ and by the divide powers of the elements $\frac{1}{2}X^\pm_{2r+1}$, adapting certain straightening relations already studied in the case of $\tuz$ (see for example Lemma \ref{nuoveadd}, \ref{muzubuzubo} and Proposition \ref{zeropiubarra}) we automatically deduce the structure of $\buz\cap\tu^\pm$. The heart of the problem is thus reduced to describe $\buz\cap \tu^{0,+}$. Here we introduce new elements that is: $ \bar h_{2r}$ (see Definition \ref{barradef}), then thanks to Section \ref{mixedsection} (Theorems \ref{lambdabaseth} and \ref{polbasis}), we can prove that is an integral form but not longer an algebra of polynomials. For this reason we decided to study $\muz$, that is obtained by $\buz$ adding extra elements $\check h_{r}$ (see Definition \ref{barradef}) in order to have a polynomial structure in the imaginary components.\\ 
In Subsection \ref{chapteraqd} we present the case of $A_{2n}^{(2)}$. In the first part introduce general definitions (see Definition \ref{aquattrodef}), in particular we devote care to the description of the root system and the related group of automorphisms $T$ generated by the $\tau_i$ (see Notation \ref{rootsystemadueennedue}), also we highlight the presence of certain embeddings, namely a copy of $A_{2n-2}^{(2)}$ and $A_{n-1}^{(1)}$ (see Definition \ref{injectivemaps}). Section \ref{glue} is devoted to the case of $A_{4}^{(2)}$. In the first part we devote ourselves to the study of positive real roots from which we see that the restriction of the integral form at the first node of the diagram turns out to be a copy of $\buz$ while the restriction at the second turns out to be a copy of $\uz(A_1^{(1)})$. In Section \ref{ultimaspero} we show inductively that the study of $A_{4}^{(2)}$ leads immediately to the case of $A_{2n}^{(2)}$ with $n>2$.

\section{Commutative Integral form}\label{Commutative}
This section is devoted to the study of some commutative algebras that arise in the study of the integral forms. In Subsection \ref{Commutative} we recall some results already proved in \cite{DP}. 
In Subsection \ref{mixedsection} we  study a very particular structure, that  will play a crucial role in certain integral forms in the case of $A_{2}^{(2)}$ (see Section \ref{theothers}).

\subsection{Symmetric integral forms}\label{symmetric}
Let $\Q[h_r\mid r\in \Z_{>0}]$ be the free commutative $\Q$-algebra over the elements $\{h_r\mid r\in \Z_{>0}\}$.

\begin{definition}\label{hcappucciof}
Given $a:\Z_{>0}\to\Q$ let us define the following elements: $$\sum_{k\geq 0}\hat h^{\{ a \} }_ku^k=\hat h^{\{a\}}(u)=\exp\left(\sum_{r>0}(-1)^{r-1}\frac{a_rh_r}{r}u^r\right)\in \Q[h_r\mid r>0][[u]],$$ we denote the $\Z$-algebra generated by $\{\hat h^{\{ a \} }_k\mid k>0\}$  as follows:
\begin{align}\label{eq:zdfhp}
\Z[\hat h_k^{\{a\}}\mid k>0]=
\Z^{(sym,a)}[h_k\mid k>0]
\subseteq\Q[h_r\mid r>0].\end{align}
In the special case where $a=\uno\;$, meaning the constant sequence defined by \begin{align}
    a_r=\uno_r=1 \text{ for all } r\in\Z_{>0},
\end{align}  we omit the superscript $\uno$, thus 
$\hat h^{\{\uno\}}(u)$ is simply denoted by $\hat h(u)$ and 
$\Z^{(sym,\uno)}[h^{\{\uno\}}_r\mid r>0]$ is written as $\Z^{(sym)}[h_r\mid r>0]$ and $\Z[\hat h_r\mid r>0]$.
\end{definition}
\begin{remark}
$\Z^{(sym)}[h_r\mid r>0]$  is a polynomial algebra (see \cite{DP}, Appendix B).
\end{remark}

Here we recall the connection between the integral form $\Z^{(sym)}[h_r\mid r>0]$ of $\Q[h_r\mid r>0]$ and the homomorphisms $\lambda_m$'s for $m\in \Z_{>0}$, defined by $\lambda_m(h_r)=h_{mr}$ for all $r\in \Z_{>0}$, namely we give another $\Z$-basis of $\Z^{(sym)}[h_r\mid r>0]$, basis defined in terms of the elements $\lambda_m(\hat h_k)$'s and arising from Garland's and Mitzman's description of the integral form of the affine Kac-Moody algebras (see \cite{HG},\cite{DM} and \cite{DP}, Appendix B).

\begin{remark}
 Let's  fix $m>0$, let $\uno^{(m)}$ denote the function defined by \begin{align}
\uno^{(m)}_r=\begin{cases}m&{\rm{if}}\ m\mid r\\0&{\rm{otherwise}},\end{cases}
\end{align} thus $\hat h^{\{\uno^{(m)}\}}(-u)=\lambda_m(\hat h(-u^m))$.  
\end{remark}

\begin{theorem}\label{bun}
Let us define the following elements and subsets in $\Q[h_r\mid r>0]$:

\begin{enumerate}[label=\roman*.]
    \item $b_{{\bf{k}}}=\prod_{m>0}\lambda_m(\hat h_{k_m})$ where ${\bf{k}}:\Z_{>0}\to\N$ is finitely supported; 
\item\begin{align}\label{GarlandBase}
B_{\lambda}=\left\{b_{{\bf{k}}}\mid {\bf{k}}:\Z_{>0}\to\N\,\,{\rm{is\, finitely\, supported}}\right\};\,\,\,\,\,\,\,\,\,\,\,\,\,\,\,\,\,\,\,\,\,\,\,\,\,\,\,\,\,\,\,\,\,\,\,\,\,\,\,\,\,\,\,\,\,\,\,\,\,\,\,\,\,\,\,\,\,\,\,\,\,\end{align}
\item$\Z_{\lambda}[h_r\mid r>0]=\sum_{{\bf{k}}}\Z b_{{\bf{k}}}$ is the $\Z$-submodule of $\Q[h_r\mid r>0]$ generated by $B_{\lambda}$.
\end{enumerate} Then $\Z^{(sym)}[h_r\mid r>0]$ is a free $\Z$ module with basis $B_\lambda$. Equivalently:
\begin{enumerate}[label=\roman*.]
    \item $\Z^{(sym)}[h_r\mid r>0]=\Z_\lambda[h_r\mid r>0]$,
    \item $B_\lambda$ is linearly independent.
\end{enumerate}
 \end{theorem}
\begin{remark}\label{lambdastabilita}
    In particular we have that $\Z[\hat h_r\mid r>0]$ is $\lambda_m$ stable for any $m\in\Z_{>0}$
\end{remark}

\begin{definition}\label{binom}
Let us consider the following $\Q$-algebra homomorphism:
\begin{align}
    &b:h_r\mapsto x,\\
    &dp:h_r\mapsto \delta_{r,1}x,
\end{align}
then 
\begin{align}
    &b:\hat h_r\mapsto \binom{x}{r}:=\frac{x(x-1)\dots (x-r+1)}{r!},\\
    &dp:\hat h_r\mapsto x^{(n)}:=\frac{x^n}{n!},
\end{align} that are called respectively
the $n$-th binomials  and $n$-th divided powers of $x$.
The $\Z$-algebras of the divided powers and of the binomials of $\Q[x_i\mid i\in I ]$ (with respect to the generators $\{x_r\mid r\in I\}$) are  respectively 
\begin{align}
    U_\Z^{(Div)}[x_i\in I]=\Z( x_i^{(r)}\mid r\in \N,i\in I),\\
    U_\Z^{(Bin)}[x_i\in I]=\Z( \binom{x_i}{n}\mid r\in \N,i\in I).
\end{align}
It is well known that $U_\Z^{(Div)}[a_i\in I]$ and $U_\Z^{(Bin)}[a_i\in I]$ are integral forms of $\Q[a_i\mid i\in I ]$ (see \cite{DP}).
\end{definition}

\subsection{A "mixed symmetric" integral form}\label{mixedsection}
Given $a:\Z_{>0}\rightarrow\Q$ we have seen in (see \cite{DP} Remark 1.26 and Propositions 1.23 and 1.24) when $\hat h^{\{a\}}(u)\in \Z[\hat h_k\mid k>0][[u]]$:
\begin{proposition}\label{condizione}
 Given a sequence $a:\Z_{>0}\rightarrow \Q$, then
 \begin{align}\label{eq:condizione}
     \hat h_{k}^{\{a\}}\in \Z[\hat h_l\mid l>0] \; \forall k>0\Leftrightarrow p^s\vert (a_{mp^{s}}-a_{mp^{s-1}}) \,\forall m,p,s\in \Z_{>0} \text{ with } p \text{ prime and } (m,p)=1.
 \end{align}
\end{proposition} 
But what happens if $\hat h^{\{a\}}(u)\not\in \Z[\hat h_k\mid k>0][[u]]$ and we consider the $\Z$-algebra generated by $\{\hat h_k,\hat h^{\{a\}}_k\mid k>0\}$? Is it still an integral form of $\Q[h_r\mid r>0]$?  Is it still an algebra of polynomials?  %In particular the first part is devoted to study the case when $a$
%In particular the first part is devoted 
%In this part we study another special case case of the generalized $(Sym)$-operator, as before this study is part of the preliminary material needed for the investigation of the integer form of $A_{2}^{(2)}$ (see Chapter \ref{theothers}).
Here we answer to the previous questions in the case when $a=\frac{1}{2}\uno^{(2)} $ and we describe the structure of this algebra (that we will denote by $\Z^{(mix)}[h_r\mid r>0]$) in two different ways in Theorems \ref{lambdabaseth} and \ref{polbasis}. %The last part is devoted to show that a specific sequence ($c(r)=2^{r-1}$) belongs to the latter algebra.
\begin{definition}\label{barradef}
    Using the notations introduced in Notation \ref{hcappucciof}
    %, let us define the sequences $$\frac{1}{2}\uno^{(2)}, \frac{1}{2}\uno:\Z_{>0}\rightarrow\Q,$$ more precisely  \begin{align}\frac{1}{2}\uno^{(2)}(r)=\begin{cases}0 &\text{ if } r \text{ is odd},\\1 &\text{ if } r \text{ is even}.\\\end{cases}\end{align} and \begin{align}\frac{1}{2}\uno(r)=\frac{1}{2} \text{ for all } r.\end{align}
    let us set $\bar h(u)=\sum_{k\ge0} \bar h_{k}u^k=\hat h^{\frac{1}{2}\uno^{(2)}}(u)$ and $\check h(u)=\sum_{k\ge0} \check h_{k}u^k=\hat h^{\frac{1}{2}\uno}(u)$.
\end{definition}
\begin{definition}\label{mixdef}
    Define $\Z^{(mix)}[h_r\mid r>0]$ to be the $\Z$-subalgebra of $\Q[h_r\mid r>0]$ generated by $\{\hat h_{r},\bar h_{r}\mid r>0\}$.
\end{definition}

\begin{remark}\label{rembarceckbo}
    $\bar h(u)\in \Q[h_{2r}\mid r>0]$ and $\bar h_{2r+1}=0$ $\forall r>0$. More precisely 
    \begin{align}
        \Z[\bar h_{2r}\mid k>0]=\Z^{(sym)}[\frac{h_{2r}}{2}\mid r>0]
    \end{align} and \begin{align}
        \bar h(u^2)=\lambda_2(\hat h^{\frac{1}{2}}(u^2))=\lambda_2(\check h(u^2))=\check h(u) \check h(-u).
    \end{align}
\end{remark}

\begin{lemma}\label{hatbarincec}
 \begin{enumerate}
     \item $\hat h(u)\not \in \Z[\bar h_{2r}\mid r>0],$
     \item $\bar h(u)\not \in \Z[\hat h_{r}\mid r>0],$
     \item $\hat h(u), \bar h(u)\in\Z [\check h_k\mid k>0]=\Z^{(sym)}[\frac{h_r}{2}\mid r>0],$
     \item $\Z^{(mix)}[h_r\mid r>0]\subseteq \Z [\check h_{k}\mid r>0].$
 \end{enumerate}   
\end{lemma}
\begin{proof}
    1. follows directly from Remark \ref{rembarceckbo}. 2) follows form Proposition \ref{condizione}, since $\frac{1}{2}\uno^{(2)}$ does not satisfy condition \eqref{eq:condizione}. 3. and 4. follow directly from Definitions \ref{barradef} and \ref{mixdef} and Remarks \ref{lambdastabilita} and \ref{rembarceckbo}.
\end{proof}

\begin{remark}\label{chechdiverso}
    Let $V$ be the $\Q$-vector subspace of $\Q[h_r\mid r>0]$ with basis $\{ h_r\mid r>0\}$. Then 
    \begin{align}
        &\Z^{(mix)}[h_r\mid r>0]\cap V=\Z\langle h_{2r-1},\frac{h_{2r}}{2}\mid r>0\rangle,\\
        &\Z[\check h_r\mid r>0]\cap V=\Z\langle \frac{h_{r}}{2}\mid r>0\rangle.
    \end{align}
    Thus, $\Z^{(mix)}[h_r\mid r>0]\subsetneq\Z[\check h_{r}\mid r>0]$.
\end{remark}

\begin{proposition}\label{notpol}
$\Z^{(mix)}[h_r\mid r>0]$ is not a polynomial algebra in homogeneous variable. Specifically, there is no sequence $a:\Z_{>0}\rightarrow \Q$ such that  $\Z^{(mix)}[h_r\mid r>0]=\Z[\hat h^{\{a\}}_k\mid k\ge 0]$.
\end{proposition}
\begin{proof}
    $\Z^{(mix)}[h_r\mid r>0]$ is a graded algebra with $\text{deg}(h_r)=r$ for all $r>0$, that is \begin{align}
        \Z^{(mix)}[h_r\mid r>0]=\bigoplus_{d\ge 0}\Z^{(mix)}[h_r\mid r>0]_d.
    \end{align} We have $\Z^{(mix)}[h_r\mid r>0]_1=\Z h_1$ and \begin{align}
        \Z^{(mix)}[h_r\mid r>0]_2=\Z\langle h_1^2, \hat h_2=\frac{1}{2}(h_1^2+h_2), \bar h_2=\frac{1}{2} h_2\rangle=\Z\langle \frac{1}{2} h_1^2,\frac{1}{2} h_2\rangle
    \end{align} which implies that $h_1^2$ does not belong to any $\Z$-basis of $\Z^{(mix)}[h_r\mid r>0]_2$. 
\end{proof}
Even though $\Z^{(mix)}[h_r\mid r>0]$ is not a polynomial algebra we aim to prove that it is though an integral form of $\Q[h_r\mid r>0]$, by exhibiting a $\lambda$-Garland type $\Z$-basis of $\Z^{(mix)}[h_r\mid r>0]$. Additionally we will provide a polynomial-like basis of this $\Z$-algebra.% In the following $k:\Z_{>0}\rightarrow \N$ will denote a finitely support function.

Recall that \begin{align}
    B_\lambda=\{ b_k=\prod_{m>0} \lambda_m (\hat h_{k_m})\vert k:\Z_{>0}\rightarrow \N \text{ is finitely supported}\}
\end{align} is a basis of $\Z[\hat h_k\mid k>0]$. 
\begin{definition}\label{lambdabase}
As in Theorem \ref{bun} we define the following elements and set:
\begin{align}
        &\bullet b_k'=\prod_{m>0, m \text{ odd } } \lambda_m (\hat h_{k_m})\prod_{m>0, m \text{ even } }  \lambda_m (\check h_{k_m}),\, \text{ for }k:\Z_{>0}\rightarrow\N \text{  finitely supported}, \\
        &\bullet B_\lambda'=\{ b_k'\mid k:\Z_{>0}\rightarrow\N \text{ is finitely supported}\},\\
        &\bullet\Z_\lambda'[h_r\mid r>0]= \Z\text{-linear span of } B_\lambda'.
    \end{align}
\end{definition}
\begin{remark}
\begin{enumerate}[label=\roman*.]
    \item  $b_k'\in \Z^{(mix)}[h_r\mid r>0]$, %that is $\Z_\lambda'[h_r\mid r>0]\subseteq \Z^{(mix)}[h_r\mid r>0]$: this follows from the fact that $\lambda_2(\check h_k)=\bar h_{2k}$ $\forall k\ge 0$ and from the $\lambda$-stability of $\Z^{(sym)}$.
    \item $\hat h_k, \bar h_k\in \Z_\lambda'[h_r\mid r>0]$ $\forall k\ge0$: indeed $\hat h_k=\lambda_1 (\hat h_k)$ and again $\lambda_2(\check h_k)=\bar h_{2k}$.
\end{enumerate}
\end{remark}
\begin{theorem}\label{lambdabaseth}
    $\Z^{(mix)}[h_r\mid r>0]=\Z_\lambda'[h_r\mid r>0]$ is  and integral form of $\Q[h_r\mid r>0]$ and $B'_\lambda$ is $\Z$-basis of $\Z^{(mix)}[h_r\mid r>0]$. 
\end{theorem}
\begin{proof}
    Thanks to previous remark, in order to prove that $\Z^{(mix)}[h_r\mid r>0]=\Z_\lambda'[h_r\mid r>0]$ it is enough to show that $\Z_\lambda'[h_r\mid r>0]$ is closed by multiplication. Notice that $\forall m>0$ $\lambda_{2m}(\hat h(u))\in \Z[\bar h_{2r}\mid r>0][[u]]$ since $\hat h(u)\in \Z [\check h_k\mid k>0] $ and $\lambda_2 (\check h_k)=\bar h_{2k}$. Then the fact that $\{b_k\}$ is a $\Z$-basis of $\Z[\hat h_k\mid k>0]$ implies the following facts, which imply the claim:\begin{enumerate}[label=\roman*.]
        \item $\prod_{m>0, m \text{ is even}}\{ \lambda_m(\check h_k)\mid k:\Z_{>0}\rightarrow \N\text{ is finitely supported }\}$  is a $\Z$-basis of $\Z[\bar h_{2k}\mid k>0 ]$;
        \item $b_k={\prod_{m>0, m \text{ is odd}}}\lambda_m(\hat h_k)\cdot b_k^{\text{even}}$ with $b_k^{\text{even}}\in \Z[\bar h_{2k}\mid k>0]$.
        \item $b'_{k'}, b'_{k''}=\prod_{m>0, m\text{ is odd} } \lambda_{m}(\hat h_{k'_m})\lambda_{m}(\hat h_{k''_m})\cdot \bar b' \bar b''$ with $\bar b',\bar b'' \in \Z[\bar h_{2k}\mid k>0]$ is a $\Z$-linear combination of elements of the form $\prod_{m>0, m\text{ is odd}}\lambda_{m}(\hat h_{k_m})\bar b$ with $\bar b\in \Z[\bar h_{2k}\mid k>0]$.
    \end{enumerate} Finally it is obvious that the $\Q$-span of $\Z'_\lambda[h_r\mid r>0]$ is $\Q[h_r\mid r>0]$ and the linear independence of $B'_\lambda$ now follows by dimension considerations:
    \begin{align}
        &\#\{b_k'\mid \text{ deg}(b_k')=d\}=
        \#\{k:\Z_{+}\rightarrow\N\mid \sum_{m>0} mk_m=d\}=\\
        &\# \{b_k\mid \text{ deg}(b_k)=d\}= \text{dim} \Q[h_r\mid r>0]_d.
    \end{align}
\end{proof}
\begin{corollary}
    $\Z^{(mix)}[h_r\mid r>0]$ is a $\Z[\bar h_{2k}\mid k>0]$-free module with basis \begin{align}
        \{\prod_{m>0}\lambda_{2m-1}(\hat h_{k_m})\mid k:\Z_{>0}\rightarrow\N \text{ is finitely supported}.\}
    \end{align}
\end{corollary}
We now give also a "polynomial-like" $\Z$-basis of $\Z^{(mix)}[h_r\mid r>0]$, before let us recall the following classical result (see \cite{EULER}):
\begin{theorem}[Euler]\label{euler}
    The number of partitions of a positive integer $n$ into distinct parts is equal to the number of partitions of $n$ into odd parts.
\end{theorem}
%\begin{proof}Let us denote by $D(n)$ and by $O(n)$ respectively the number of partitions of $n$ into distinct parts and the number of partitions of $n$ into odd parts, then it is immediate to see that: \begin{align}   \sum_{n\ge 0} D(n)x^n=\prod_{i\ge 1}(1+x^{i}),\\   \sum_{n\ge 0} O(n)x^n=\prod_{i\ge 1}\frac{1}{1-x^{2i-1}}.\\\end{align} The claim follows observing that \begin{align}\prod_{i\ge 1}(1+x^{i})=\prod_{i\ge 1}\frac{1-x^{2i}}{1-x^{i}}=\prod_{i\ge 1}\frac{1}{1-x^{2i-1}}.\end{align}\end{proof}

\begin{lemma}
 The following identities hold in $\Q[h_r\mid r>0][[u]]$:
    \begin{align}
&\label{eq:hatbar}\lambda_2(\hat h(u^2))=\hat h(u)\hat h(-u)=\bar h(u^2)^2,\\&
\label{eq:cappucciobarra}\sum_{s=0}^{2r}\hat h_{2r-s}\hat h_{s}(-1)^s=\sum_{s=0}^{r}\bar h_{2r-2s}\bar h_{2s}.\\
    \end{align}
\end{lemma}
 \begin{proof}
    Equation \eqref{eq:hatbar} follows directly from Definition \ref{barradef} and  Notation \ref{hcappucciof}, Equation \eqref{eq:cappucciobarra} follows from Equation \eqref{eq:hatbar} and \cite{DP},Proposition 1.19.
\end{proof}

\begin{theorem}\label{polbasis}
    $\Z^{(mix)}[h_r\mid r>0]$ is a $\Z[\bar h_{2r}\mid r>0]$-free module with basis \begin{align}
        \{\prod_{k>0}\hat h_{k}^{\epsilon_k}\mid \epsilon:\Z_{>0}\rightarrow \{0,1\} \text{ is finitely supported} \}.
    \end{align}
    Equivalently 
    \begin{align}
         B_{q.pol}=\{ \prod_{k>0}\hat h_k^{\epsilon_k}\prod_{k>0}\bar h_k^{d_k}\mid \epsilon:\Z_{>0}\rightarrow \{0,1\} \text{ and } d:\Z_{>0}\rightarrow \N\text{ are finitely supported} \}
    \end{align} is a $\Z$-basis of $\Z^{(mix)}[h_r\mid r>0]$.
\end{theorem}
\begin{proof}
    We prove that the $\Z[\bar h_{2r}\mid r>0]$-span of $\{\prod_{k>0}\hat h_k^{\epsilon_k}
    %\prod_{j>0}\bar h_{j2}^{d_j}
    \mid \epsilon\in\{0,1\}\}$ is stable by multiplication by the $\hat h_l$'s. More precisely, we prove by induction on $N=\sum k\epsilon_k$ that $\hat h_l \prod_{k}\hat h_k^{\epsilon_k}$ is in the $\Z$ span of $B_{q.pol}$. If $N=0$ the claim is obvious. Let us assume that $N>0$ and the claim holds for all $\tilde N<N$. If $l\neq k$ for all $k$ such that $\epsilon_k=1$ (or equivalently $\epsilon_l=0$) the claim is obvious. So suppose that  there $\epsilon_l=1$. Let us consider the monomial  $\hat h_l^{2}\hat p$ with $\hat p=\prod_{k\neq l}\hat h_k^{\epsilon_k}$ and remark that $\text{deg}(\hat p)=N-l$.  
%Remark that if $\text{ deg}(\bar q)>0$ then the claim follows by induction hypothesis on $\hat h_l^{2}\hat p $, let us therefore assume that $\text{ deg}(\bar q)=0$. 
Using relation \eqref{eq:cappucciobarra} we have that \begin{align}
    \hat p\hat h_{l}^2=\hat p(2\sum_{j=1}^{l}(-1)^{j+1}\hat h_{l+j}\hat h_{l-j}+(-1)^l\sum_{j=0}^{l}\bar h_{2j}\bar h_{2l-2j}),
\end{align} since the right summand is in the $\Z$-span of $B_{q.pol}$,
 let us focus on the monomials of the form $\hat p \hat h_{l-j}\hat h_{l+j}$ for some $j\ge 1$. Since $\text{deg}( \hat p)<N$, $\hat h_{l-j} \hat p$ is in the $\Z[\bar h_{2r}\mid r>0]$-span of $\{\prod_r\hat h_r^{\epsilon_r}   \mid \sum r\epsilon_r\le N-l+l-j=N-j<N\}$ so that by the induction hypothesis $\hat h_{l+j}\hat h_{l-j}\hat p$ lies in the $\Z[\bar h_{2r}\mid r>0]$-span of $\{\prod_r\hat h_r^{\epsilon_r}\mid \epsilon_r\in\{0,1\}\}$ 
 %If $\hat h_{l-j}$ appears in $\hat p$  by induction hypothesis on $\hat p \hat h_{l-j}$ we can restrict to assume that the repeated element is $\hat h_{l+j}$, hence we are restricted to consider monomials of the form $\hat p_1 \hat h_{l+j}^2$ for some $\hat p_1$ that satisfy the claim an in which $\hat h_{l+j}$ does not appear. 
%We can operate the same substitution getting element of the form $\hat p_n \hat h_{l+j_1+j_2+\dots+j_n}^2$ where each $j_i>0$ and for some $\hat p_n$ that satisfy the claim an in which $\hat h_{l+j_1+j_2+\dots+j_n}$ does not appear, since all the relations are homogeneous, we can assume that a some point $\text{deg} (\hat p_{n} \hat h_{l-(j_1+\dots+j_n)})<\text{deg} (\hat h_{l+j_1+\dots+j_n})$ and hence the claim follows by induction hypothesis on $\hat p_n h_{l-(j_1+\dots+j_n)}$ since $\hat p_n h_{l-(j_1+\dots+j_n)}$ could not contains $\hat h_{l+j_1+\dots+j_n}$.\\
We are left to prove that $ B_{q.pol}$ is linearly independent.
Let us observe that the elements of $ B_{q.pol}$ of degree $d$  are clearly indexed by the pairs of partitions $(\lambda',\lambda'')$ such that  $\lambda'\vdash n'$ consist only of not repeating integers, $\lambda''\vdash n''$ consist of even integers and  $n'+n''=d$; on the other hand the elements  of $ B'_{\lambda}$ of degree $d$  are clearly indexed by the pairs of partitions $(\tilde \lambda',\lambda'')$ such that  $\tilde\lambda'\vdash n'$ consist only of odd integers, $\lambda''\vdash n''$ consist of even integers and  $n'+n''=d$. It follows from Euler's theorem (see Theorem \ref{euler}) on partitions that these sets have the same cardinality.
%this result obviously define a bijection between $\tilde \lambda'\vdash n'$ and $\lambda'\vdash n'$.
\end{proof}

%In order tho show that $h_r^{\{c\}}\in \Z^{(mix)}[h_r\mid r\in \Z_{>0}]$ let us notice the following fact: 

In the last part of this section we want to prove that $\hat  h_k^{\{c\}}\in\Z^{(mix)}[h_r\mid r>0]$ for all $k>0$ where
 $c:\Z_{>0}\rightarrow\Q$ is the sequence defined by \begin{align}\label{eq:circledef}
    c_r={2^{r-1}}.
\end{align}
\begin{remark} Let us remark that:
 \begin{align}\label{eq:circonohat}
     \Z[\hat h_r^{\{c\}}\mid r>0]\nsubseteq \Z[\hat h_r\mid r>0],\\
     \Z[\hat h_r^{\{c\}}\mid r>0]\nsubseteq \Z[\bar h_r\mid r>0].\\
 \end{align} 
 The first condition follows by Proposition \ref{condizione} with $(m,p,s)= (1,2,1)$, since  $2\not \vert 2^{2-1}-2^{1-1}$, the second is trivial since $h_1\not\in\Z[\bar h_{2r}\mid r>0]\subseteq \Q[ h_{2r}\mid r>0]$.
\end{remark}

\begin{lemma}\label{quellochenonvolevo}
 Let $\star$ denote the convolution product, $\mu$ the Mobius function and $l:\Z_{>0}\rightarrow\Z$ then the following hold:
\begin{enumerate}
\item if $l(2r+1)=0$ for $\forall r\in \Z_{\ge0}\Rightarrow$ $(\mu\star l)(2r+1)=0$ $\forall r\in \Z_{\ge0},$
\item $\hat h_r^{\{l\}}\in\Z[\hat h_{k}\mid k>0]\Leftrightarrow r\vert(\mu\star l)(r)$,
\item $\hat h_r^{\{l\}}\in\Z[\bar h_{2k}\mid k>0]\Leftrightarrow(2r)\vert2(\mu\star l)(2r) \text{ and } l(2r+1)=0$,
\item $\hat h_r^{\{l\}}\in\Z^{(mix)}[ h_{k}\mid k>0]\Leftrightarrow 
     (2r+1)\vert(\mu\star l)(2r+1) \text{ and } (2r)\vert2(\mu\star l)(2r)
     $.
\end{enumerate}
\end{lemma}
\begin{proof}

\begin{enumerate}
\item \begin{align}
   (\mu\star l)(2r+1)=\sum_{d\vert (2r+1)}\mu(\frac{2r+1}{d})l(d)=0; 
\end{align}
\item See Proposition \cite{DP}, 1.24;
\item From 2. it follows that \begin{align}
    \hat h_r^{\{l\}}\in\Z[\check h_{r}\mid r>0]\forall r>0  \Leftrightarrow 2r\vert(\mu\star l)(r)\end{align}
    thus \begin{align}
    \hat h_r^{\{l\}}\in\Z[\bar h_{r}\mid r>0]\forall r>0  \Leftrightarrow 2r\vert(\mu\star l)(r) \text{ and } l(2r+1)=0 
\end{align}
\item Let $m$ and $n$ respectively the even and the odd part of $l$, namely $m(2r+1)=l(2r+1)$, $n(2r)=l(2r)$  and  $m(2r)=0=n(2r+1)$, thus 
$\hat h_r^{\{l\}}=\hat h_r^{\{m\}}\hat h_r^{\{n\}}$, the claim follows from 1.,2. and 3..
\end{enumerate}
\end{proof}

\begin{theorem}\label{toglicerchietti}
$\hat h_{k}^{\{c\}}\in \Z^{(mix)}[h_r\mid r\in \Z_{>0}]$    
\end{theorem}
\begin{proof} 
Let us denote by $f$ the double of $c$, namely $f(r)=2^r$, and let $r=\prod_{i=1}^{k}p_i^{a_i}$ be the decomposition of $r$ in prime factors. $\mu$ and $f$ are weak multiplicative, that is they are multiplicative on the coprime factors, then $\mu\star f$ is weak multiplicative.
\begin{align}
(f\star\mu)(r)=\prod_{i=1}^{k}(\sum_{d\vert p_i^{a_i}}f\big(\frac{p_i^{a_i}}{d}\big) \mu(d))=
\prod_{i=1}^{k}(f(p_i^{a_i}) \mu(1)+f(p_i^{a_i-1}) \mu(p_i))=\prod_{i=1}^{k}(2^{p_i^{a_i}}-2^{p_i^{a_i-1}})
\end{align}
Notice that:
\begin{align}
    2^{p_i^{a_i}}-2^{p_i^{a_i-1}}=2^{p_i^{a_i-1}}(2^{p_i^{a_i-1}(p_i-1)}-1).
\end{align} If $p_i$ is odd the first factor is even and the second is a multiple of $p_i^{a_i}$ by  Euler's Theorem and hence is divisible by $2p_i^{a_i}$, that is $(2r+1)\vert (c\star \mu)(2r+1)$. If $p_i=2$ the first factor is divisible by $2^{a_i}$ because ${2^{a_i-1}}\ge a_i$ for any $a_i\ge 1$, that is $c(2r)\vert 2(c\star \mu)(2r)$. The claim follows from Lemma \ref{quellochenonvolevo}.
\end{proof}

\section{Affine Kac-Moody Algebras}\label{AffineKM} This section is organized as follows.
In the Subsection \ref{setup} we fix the notation and we recollect general results on Affine Kac-Moody algebras, we systematically refer to \cite{Bo}, \cite{Kac} and \cite{Da}. In the Subsection \ref{MitzmaneGarland} we recall the results on intergal forms due to Garland and Mitzman (see Theorem \ref{KMintegral}) and we compare them with $\uz$ (see Theorem \ref{comparison})

\subsection{Setup and Notation}\label{setup}
Let $I=\{0,\dots,n\}$ and $I_0=\{1,\dots,n\}$.
Let $A=(a_{i,j})_{i,j\in I}$ be an finite or affine Cartan Matrix. 
Let $D$ be the diagonal matrix that symmetrize $A$, chosen such that $\min\{d_i\mid i \in I\}=1$ and $\const=\max\{d_i\mid i \in I\}$.
It is well known that affine Cartan Matrix are classified by
$ (X_{\tilde n},k)$ where $X_{\tilde n}$ finite Lie algebra, $k=\text{ord}(\chi)$  and $\chi$ is a Dynkin diagram automorphism.
We denote $X^{(k)}_{\tilde n}$ by
the Affine Kac Moody algebra  associated with $A$.
It is well known that $X^{(k)}_{\tilde n}$ admits two main presentation, namely,
the Loop presentation (see Definition \ref{KacMoody}) and the Kac-Moody presentation (see Definition \ref{KacMoody2}) which we will briefly recall.

\begin{definition}\label{KacMoody2}
$X^{(k)}_{\tilde n}$ is the Lie algebra generated by 
$\{e_{i},f_{i},h_{i}\mid i\in I\}$ with relations:
\begin{align}
    &[e_{i},f_j]=\delta_{i,j} h_i;\\
    &[h_i,e_{j}]=a_{i,j}e_{j};\\
    &[h_i,f_{j}]=-a_{i,j}f_{j};\\
    &ad_{e_i}^{1-a_{i,j}}(e_j)=0=ad_{f_i}^{1-a_{i,j}}(f_j) \quad\text{ if } i\neq j.
\end{align}
\end{definition}

\begin{definition}\label{KacMoody}
  $X^{(k)}_{\tilde n}$ is the Lie algebra generated by $\{x^+_{i,r},x^-_{i,r},h_{i,r}, c\mid i \in \{1,\dots,n\},\,
\tilde d_i\vert r\in \Z\} $ with relations:
\begin{align}
    &[c,\cdot]=0;\\
    &[h_{i,r},h_{j,s}]=
    r\delta_{r+s,0}\frac{a_{i,j;r}}{d_j}\const c;\\
    &[x_{i,r}^+,x_{j,r}^-]=\delta_{i,j}(h_{i,r+s}+r\delta_{r+s,0}\frac{\const c}{d_j});\\
    &[h_{i,r},x^\pm_{j,s}]=\pm a_{i,j;r}x^\pm_{j,r+s};\\
    &[x^\pm_{i,r},x^\pm_{i,s}]=0 &\text{ if } (X_{\tilde n}^{(k)},d_i)\neq (A_{2n}^{(2)},1) \text{ or }  r+s \text{ is even};\\
    &[x^\pm_{i,r},x^\pm_{i,s}]+[x^\pm_{i,r+1},x^\pm_{i,s-1}]=0 &\text{ if } (X_{\tilde n}^{(k)},d_i)= (A_{2n}^{(2)},1) \text{ and } r+s \text{ is odd};\\
    &[x^\pm_{1,r},[x^\pm_{1,s},x^\pm_{1,t}]]=0 %\text{ if } \hat\g^\chi= A_{2n}^{(2)};
    \\
    &(\text{ad} x^\pm_{i,r})^{1-a_{i,j}}(x^\pm_{j,s})=0 &\text{ if }i\neq j.
    \end{align}
    Where; \begin{align}
    a_{i,j;r}=
    \begin{cases}
    2(2+(-1)^{r}) &\text{ if } 
 i=j,\, d_i=1 \text{ and } X_{\tilde n}^{(k)}= A_{2n}^{(2)},\\
        a_{i,j}&\text{ otherwise};
        \end{cases}
        \end{align} and  \begin{equation}
    \tilde d_i=\begin{cases}
1 &\text{ if  } k=1 \text{ or }X_{\tilde n}^{(k)}=A_{2n}^{(2)},\\
    d_i &\text{ otherwise}.\end{cases}
\end{equation}
\end{definition}
%\begin{remark}    $\{x^+_{i,\tilde d_i r},x^-_{i,\tilde d_i r}\mid i \in I_0,\, r\in \Z\}$ generates $X_{\tilde n}^{(k)}$.\end{remark}
Associated with $X_{\tilde n}^{(k)}$ is a finite dimensional simple Lie algebra $X_{ n}$, which corresponds to  $A_0=(a_{i,j})_{i,j\in I_0}$.
Let $\Phi$ and  $\Phi_0$  be the Root systems of $X_{\tilde n}^{(k)}$ and $X_{ n}$.
Denote the set of simple roots of $\Phi$ and $\Phi_0$ by respectively
$\Delta=\{\alpha_0,\alpha_1,\dots,\alpha_n\}$ and $\Delta_0=\{\alpha_1,\dots,\alpha_n\}$. 
Let $Q=\bigoplus_{i\in I} \Z \alpha_i$ and
 $Q_0=\bigoplus_{i\in I_0} \Z \alpha_i$ be the root lattice of
 respectively $X_{\tilde n}^{(k)}$ and $X_{ n}$.
 Denote by $W_0$ and $W$ the Weyl groups of respectively $X_{\tilde n}^{(k)}$ and $X_{ n}$.
The $W_0$ -invariant bilinear form $(\cdot\vert\cdot)$ on $Q_0$, which induces a positive definite scalar product on $\R\otimes_\Z Q_0$ and induces a positive semidefinite symmetric bilinear form on $\R\otimes_\Z Q$ and has kernel generated by $\delta=\alpha_0+\theta$ where $\theta\in Q_0$.
Let $P=\bigoplus_{i\in I_0}\Z\omega_i\subseteq \R \otimes_\Z Q_0$ be the weight lattice, where $\forall i \in I_0$
$\omega_i$ is defined by $(\omega_i\vert \alpha_j)=\tilde d_i \delta_{i,j}$ $\forall  j \in I_0$; $Q_0$ naturally embeds in $P$, which provides a $W$-invariant action on $Q$ by $x(\alpha) = \alpha - (x\vert \alpha)\delta \forall x \in P , \alpha \in  Q$. $\hat W=P\rtimes W_0$ is
called the extended Weyl group of $X_{\tilde n}^{(k)}$,

The root system $\Phi$ divides into two parts: the real $\Phi^{re}$ and imaginary roots $\Phi^{im}=\{m\delta\mid 0\neq m\in \Z\}$.  
It is possible to describe the $\Phi^{re}$ in terms of $\Phi_0$, as follows:
\begin{align}\label{eq:affinerootsystem}
    &\Phi^{re}=\begin{cases}
        \{ \alpha+m \delta\mid \alpha\in \Phi_0,\,m\in\Z\} &\text{ if } k=1,\\
        \{ \alpha+m \delta\mid \alpha\in \Phi_0,\,m\in\Z\}\cup
        \{ 2\alpha+(2m+1) \delta\mid (\alpha,\alpha)=2,\,m\in\Z\}
        &\text{ if } X_{\tilde n}^{(k)}=A_{2n}^{(2)},\\
        \{ \alpha+(\alpha,\alpha)m \delta\mid \alpha\in \Phi_0,\,m\in\Z\} &\text{ otherwise}.
    \end{cases}
\end{align}

\begin{definition}\label{tauiloop}
For all $i\in I$, let us define the following automorphisms of $X_{\tilde n}^{(k)}$:
\begin{align}
&\tau_i=\exp(\text{ad} e_i )\exp(-\text{ad}{f_{i}})\exp(\text{ad}e_{i}).\end{align}
Denote by $T_0$ the group generated by the $\tau_i$'s for $i\in I_0$, in particular $T_0$ is an automorphism group of $X_{n}$. 
\end{definition} 
\begin{definition}\label{rootsvectors}
Let $\Phi_0^+$ and $\Phi_0^-$ respectively the positive and the negative roots of $\Phi_0$ 
Let us fix a reduced expression $w$ on the longest element of $W_0$, that is $w_0=s_{j_1}\dots, ,s_{j_k}$.
Let $\beta_0\in \Phi_0^{+}$, then there exist $1\le l\le k$ such that $\beta_0=\sigma_{j_1}\dots\sigma_{j_{l-1}}\alpha_{j_{l}}$. We denote by $x_{\beta_0,m}^{\pm}=\tau_{j_1}\dots\tau_{j_{l-1}}(x_{j_l,m}^{\pm})$. If $\beta=2\beta_0+(2m+1)\delta$, in particular we have that $(\beta_0,\beta_0)=(\alpha_1,\alpha_1)$ hence there exist $w\in W_0$ such that $w(\alpha_1)=\beta_0$, we denote by $X_{\beta_0,2m+1}^{\pm}=\tau_{i_1}\dots\tau_{{i-l}}([x^\pm_{1,0},x^\pm_{1,2m+1}])$ if $w=\sigma_{i_1}\dots\sigma_{{i_l}}$.  
\end{definition}

\begin{remark}\label{isomorphism}
The equivalence between the two presentations of $X_{\tilde n}^{(k)}$ is defined as follows: \begin{align}
&e_i\rightarrow 
\begin{cases}
 x^+_{i,0}  &\text{ if } i\neq 0\\
 x^-_{\theta,1}  &\text{ if } X_{\tilde n}^{(k)}\neq A_{2n}^{(2)} \text{ and } i= 0\\
 \frac{1}{4}X^-_{\theta,1}  &\text{ if } X_{\tilde n}^{(k)}= A_{2n}^{(2)}\text{ and } i= 0\\
\end{cases}
&f_i\rightarrow 
\begin{cases}
 x^-_{i,0}  &\text{ if } i\neq 0\\
 x^+_{\theta,1}  &\text{ if } X_{\tilde n}^{(k)}\neq A_{2n}^{(2)} \text{ and } i= 0\\
 \frac{1}{4}X^+_{\theta,1}  &\text{ if } X_{\tilde n}^{(k)}= A_{2n}^{(2)}\text{ and } i= 0\\
\end{cases}
\end{align} 
where  $\theta$ is the highest root (respectively the highest short root) of $X_{\tilde n}$ if $k=1$ (respectively if $k\neq1$).
\end{remark}

\begin{remark}\label{embeddings}
Remark that the Loop presentation implies that there are two embeddings: \begin{align}
        \phi_i: A_{1}^{(1)}\hookrightarrow X_{\tilde n}^{(k)} \text{ if } (X_{\tilde n}^{(k)},d_i)\neq (A_{2n}^{(2)},1),\\
        \varphi : A_{2}^{(2)}\hookrightarrow X_{\tilde n}^{(k)} \text{ if } (X_{\tilde n}^{(k)},d_i)= (A_{2n}^{(2)},1).
    \end{align}
defined on the generators respectively by $x_{1,r}^{\pm}\mapsto x_{i,r}^{\pm}$ and $x_{1,r}^{\pm}\mapsto x_{1,r}^{\pm}$.
\end{remark}

\subsection{Mitzman and Garland Integral forms}\label{MitzmaneGarland}

Let $X_{\tilde n}^{(k)}$ be an affine algebra with Affine Cartan Matrix $A$. Denote by $\u$ its  universal enveloping algebra.
\begin{definition}\label{uzkm}
Define $\u_\Z^{K.M.}$ as the $\Z$-subalgebra of $\u$ generated by $\{e_{i}^{(r)},f_{i}^{(r)}\mid i\in I ,\, r\in \N\}$.    
\end{definition}

%The study of $\u_\Z$ was begun by Kostant in the 1950s in the case where $\g$ is finite and later extended to the related case by  
In the 1970s Garland \cite{HG} and Mitzman \cite{DM} in the 1980s  investigate the structure of  $\u_\Z^{K.M.}$. 
The natural question that arises at this point is: what is the relationship between $\u_\Z^{K.M.}$ and the analogous $\Z$-algebra generated by the divided powers of $x_{i,r}^{+}$ and  $x_{i,r}^{-}$, denoted by $\uz$? As we shall see these coincide except in the case $A_{2n}^{(2)}$. In that particular case, the integral form turns out to be smaller, as we will prove Subsections \ref{glue} and \ref{ultimaspero}.
\begin{definition}
The set $T\cdot\{e_i,f_i\mid i\in I\}=\{x_{\alpha}\mid \alpha\in \Phi^{re}\}$ are the root vectors considered by Mitzman and Garland.
\end{definition}
\begin{theorem}[Garland, Mitzman]\label{KMintegral}
$\u_\Z^{K.M.}$ is an integral form of $\u$. More precisely:
\begin{align}
 \u_\Z^{K.M.}\cong \u_\Z^+\otimes\u_\Z^{im,+}\otimes\u_\Z^{\h}\otimes\u_\Z^{im,-}\otimes\u_\Z^{-},   
\end{align}
where $\u_\Z^\h$ is an algebras of binomials in  the $h_i$ for $i\in I$ and
$\u_\Z^{im,\pm}\cong \otimes_{i\in I_0}\Z_\lambda[h_{i,r}\mid \pm r>0]$ (see Theorem \ref{bun}),
$\u_\Z^+$ and $\u_\Z^-$ are divided powers algebras in the real positive and real negative root vectors.
\end{theorem}

%Even though it was stated in the literature (see \cite{JM} for example), it is not clear from this description that $\uz^{im,+}$ and $\uz^{im,-}$ are algebras of polynomials, hence we decide to fill this gap giving the proof of this fact (see Theorem \ref{bun} and \cite{DP}, Appendix B for a detailed  discussion).
We want now describe the relationship between $\u_\Z$ and $\u_\Z^{K.M.}$
\begin{theorem}\label{comparison}
If $X_{\tilde n}^{(k)}\neq A_{2n}^{(2)}$, $\u_\Z$ is isomorphic to $\u_\Z^{K.M.}$.
If $X_{\tilde n}^{(k)}=A_{2n}^{(2)}$, $\u_\Z\subseteq\u_\Z^{K.M.}$.
\end{theorem}
\begin{proof}
 If $X_{\tilde n}^{(k)}\neq A_{2n}^{(2)}$   then $\alpha_0=\delta-\theta$ where $\theta\in \Phi_0$ thus there exist $i\in I_0$ and $w\in\hat W$ such that $e_0=\tau_w( f_\theta)$.
 $\uz^{K.M.}$ is $\tau_i$-stable for all $i\in I$ then there exist $w\in\hat W$ and such that $\tau_w((e_i)^{(k)})=\tau_w((x_{i,0}^{\pm})^{(k)})=(x_{i,1}^{\pm})^{(k)})$
\end{proof}

%\begin{notation}    In the following theorem where we speak about "the algebra of divided powers in the positive and negative roots vectors" we mean the $\Z$-subalgebra generated by $\{e_i^{(r)},f_i^{(r)}\mid i\in I, r \in \N\}$ which is a free $\Z$-module with basis the ordered monomials in the real root vectors  $e_\alpha$s. \end{notation}

\section{$\u(A_{2n}^{(2)})$}
From the latter section it follows that in order to conclude the study of $\uz$ we need to consider the case of $A_{2n}^{(2)}$.
Fix $I=\{1,\dots,n\}$.
As we shall see, in this case is not more true that the positive real part of the $\Z$-subalgebra by $\{(x_{i,r}^+)^{(k)},(x_{i,r}^-)^{(k)}\mid i\in I , r\in \Z, k\in \N\}$ is the $\Z$-subalgebra generated by $\Z$-subalgebra by $\{(x_{i,r}^+)^{(k)}\mid i\in I , r\in \Z, k\in \N\}$.

\begin{definition}\label{aquattrodef}
$A_{2n}^{(2)}$ (respectively $\u$) is the Lie algebra (respectively the associative algebra) over $\Q$ generated by $\{c,h_{i,r},x^\pm_{i,r},X^\pm_{1,2r+1}\vert r\in\Z\,, i\in I\}$ with relations:
\begin{align}\label{eq: aqdrel}
&[c,\cdot]=0;\\&
    [h_{i,r},h_{j,s}]=
    r\delta_{r+s,0}a_{i,j;r}\frac{2c}{d_j}\\&
    [x_{i,r}^+,x_{j,r}^-]=\delta_{i,j}(h_{i,r+s}+r\delta_{r+s,0}\frac{2c}{d_j});\\&
    [h_{i,r},x^\pm_{j,s}]=\pm a_{i,j;r}x^\pm_{j,r+s};\\&
    [x_{1,r}^\pm,x_{1,s}^\pm]=\begin{cases}
        \pm(-1)^s X_{1,r+s}^\pm \text{ if $r+s$ is odd}\\
        0 \text{ otherwise;}\end{cases}\\
        &[x_{1,r}^\pm,X_{1,2s+1}^\pm]=0;\\&
    (\text{ad} x^\pm_{i,r})^{1-a_{i,j}}(x^\pm_{j,s})=0 \text{ if }i\neq j;\\
    &[x^\pm_{i,r},x^\pm_{i,s}]=0 \text{ if } r+s \text{ is even or } i\neq1 ;\label{eq:solorpius}\\
&[x_{1,r}^+,[x_{1,r}^+,x_{2,s}^+]]=-[x_{1,r+1}^+,[x_{1,r+1}^+,x_{2,s-2}^+]]\label{eq:rrscambiasegno};
    \end{align}
where 
\begin{align}A=(a_{i,j})_{i,j=1,\dots,n,0}=
    \begin{pmatrix}
        2 & -2 & 0 & \dots& 0\\
        -1 & 2 & -1 & \dots& 0 \\
        0 & -1 & 2  & \dots& 0 \\
        0 &\ddots& \ddots & \ddots & \vdots \\
        \vdots &\ddots& -1 & 2 & -2 \\
        \vdots &\dots& 0 & -1 & 2 \\
    \end{pmatrix}\text{ if $n>2$};\quad A=(a_{i,j})_{i,j=1,0}=
    \begin{pmatrix}
        2 & -4 \\
        -1 & 2\\
    \end{pmatrix},\text{ if $n=2$};
\end{align}   
\begin{align}
\end{align}

\begin{align}   
   a_{i,j;r}=\begin{cases}
   2(2+(-1)^{r-1})&\text{ if } (i,j)=(1,1);\\
a_{i,j} &\text{ otherwise }.
\end{cases} 
\end{align}
\end{definition}
%Notice that $\{x_{i,r}^+,x_{i,r}^-\vert r\in\Z , i\in I\}$ generates $\tu$.
\begin{definition}
    Let us denote $\u^+,\u^-,\u^{0,+},\u^{0,-}$ and $\u^\h$ the subalgebras of $\u$  generated respectively by 
    \begin{align}
      &\{x_{i,r}^+\mid i\in I, r\in\Z\},\ \{x_{i,r}^-\mid i\in I, r\in\Z\},\\
    &\{h_{i,r}\mid i\in I, r\in\Z_{>0}\},\{h_{i,r}\mid i\in I, r\in\Z_{<0}\},
    \{h_{i,0}, c \mid i\in I\}
    .  
    \end{align}  and by $\u^0$ the algebra generated by $\u^{0,+},\u^{0,-}$ and $\u^\h$ 
\end{definition}
\begin{definition}\label{aqautomor}
$A_{2n}^{(2)}$ and $\u$ are endowed with the following anti/auto/ho\-mo/morphisms:\\
$\sigma$ is the antiautomorphism defined on the generators by:
\begin{align}
    &x_{i,r}^\pm\mapsto x_{i,r}^\pm,\\
    &X_{1,r}^\pm\mapsto-X_{1,r}^\pm,\\
    &h_{i,r}\mapsto h_{i,r},\\
    &c\mapsto-c;
\end{align}
$\Omega$ is the antiautomorphism defined on the generators by:
\begin{align}
    &x_{i,r}^\pm\mapsto x_{i,-r}^\mp,\\
    &X_{1,r}^\pm\mapsto X_{1,-r}^\mp,\\
    &h_{i,r}\mapsto h_{i,-r},\\
    &c\mapsto c;
\end{align}
$T$ is the automorphism defined on the generators by:
\begin{align}
    &x_{i,r}^\pm\mapsto x_{i,r\mp 1}^\mp,\\
    &X_{1,r}^\pm\mapsto -X_{1,r\mp 2}^\mp,\\
    &h_{i,r}\mapsto h_{i,-r}-r\delta_{r,0}c,\\
    &c\mapsto c.
\end{align}
\end{definition}
\begin{notation} \label{rootsystemadueennedue}
Recalling that $\Phi^{re}$, the set of real roots of $A_{2n}^{(2)}$, decompose into positive real and negative roots, $\Phi^{re}= \Phi^{re,+}\cup \Phi^{re,-}$ with the property $\Phi^{re,+}=- \Phi^{re,-}$, moreover 
$\Phi^{re,+}$ can be described as follows:
\begin{align}
\Phi^{re,+}&=\Phi_s^{re,+}\cup\Phi_m^{re,+}\cup\Phi_l^{re,+},
\end{align} where
\begin{align}
\Phi_s^{re,+}&=\{\alpha+r\delta\mid \alpha\in\Phi_{0,s}^+ , r\in \Z\},\\
\Phi_m^{re,+}&=\{\alpha+r\delta\mid \alpha\in\Phi_{0,m}^+ , r\in \Z\},\\
\Phi_l^{re,+}&=\{2\alpha+(2r+1)\delta\mid \alpha\in\Phi_{0,s}^+ , r\in \Z\},\\
\end{align} %in particular we have that: \begin{align}\Phi_s^{re}=\Phi_s^{re,+}\cup\Phi_s^{re,-}=W\cdot\alpha_1 ,\\\Phi_m^{re}=\Phi_m^{re,+}\cup\Phi_m^{re,-}=W\cdot\alpha_2,\\\Phi_l^{re}=\Phi_l^{re,+}\cup\Phi_l^{re,-}=W\cdot\alpha_0,\\\Phi^{re}=W\cdot\alpha_1\cup W\cdot\alpha_2\cup W\cdot\alpha_0.\end{align} 
%$\Phi^{im}=\Phi^{im,+}\cup\Phi^{im,-}$, $\Phi^{im,+}=-\Phi^{im,-}$ where \begin{align}    \Phi^{im,\pm}=\{\pm m\delta\mid m\in \Z_{>0}\}.\end{align}
where $\Phi_{0,s}^{+}$ and $\Phi_{0,m}^{+}$ are respectively:
\begin{align}
    &\Phi_{0,s}^{+}=\{\alpha_i+\dots+\alpha_j\mid 1\le i\le j\le n\},\\
&\Phi_{0,m}^{+}=\{2\alpha_1+\dots+2\alpha_{i}+\alpha_{i+1}+\dots+\alpha_j\mid 1\le i<j\le n\}%,\\&\Phi_{0}^+=\Phi_{0,s}^{+}\cup \Phi_{0,m}^{+}
.
\end{align}
moreover we set  $\Phi^{+}_0=\Phi_{0,s}^{+}\cup \Phi_{0,m}^{+}$.
\end{notation}
%\begin{definition}\label{defxnormali}For all $\alpha\in \Phi^{re}$ such that $\alpha\neq 2\alpha'+(2r+1)\delta$, we have that is $\alpha=w(\alpha_i)$ for some $w\in W_0$ (the finite Weyl Group generated by $\sigma_i$, $i\in I_0$), then we define $x_{\alpha,r}=\tau_{i_1}\dots\tau_{i_N}(x^+_{i,0})$ if $w=\sigma_{i_1}\dots\sigma_{i_N}$.\end{definition}

%\begin{remark}    $x_{\alpha,r}$ is defined up to sign. In particular the $\Z$-subalgebra generated by $x_{\alpha,r}^{(k)}$s is uniquely determined. %Let us fix $x_{\alpha_1+\alpha_2,0}^+=\tau_2(x_{1,0}^+)$.\end{remark}

\begin{remark}
    $ \{x_{\alpha}\mid \alpha\in \Phi^{re}\}$ is the set of Chevalley generators used by Mitzman. In particular the $\Z$-subalgebra of $\tu$ generated by $T\cdot\{(e_i)^{(k)}\mid i\in I\cup \{0\}, r\in \N\}$ is a free $\Z$-module with basis the ordered monomials in the divided powers of the ${x_{\alpha}} $'s.
\end{remark}

\begin{definition}\label{xgranditau}
    For all $i\in I$, let us define recursively the following elements: \begin{align}
X^\pm_{i,2r+1}=
\begin{cases}
    \pm[x^\pm_{1,2r+1},x^\pm_{1,0}] &\text{ if } i=1\\
 \tau_i(X^\pm_{i-1,2r+1}) &\text{ if } i>1.
\end{cases}
\end{align}
\end{definition}

\begin{definition}\label{injectivemaps}
 The following maps are Lie-algebra homomorphisms, obviously injective, inducing embeddings:
    \begin{align}\label{eq:psi}
        \bar\psi:&A_{2(n-1)}^{(2)}\rightarrow A_{2n}^{(2)}\\
        &x^\pm_{i,r}\mapsto x^\pm_{i,r}\\
        &h_{r}\mapsto h_{i,r}\\
        &c\mapsto c\\\\\label{eq:tildepsi}
        \tilde\psi:& A_{n-1}^{(1)}\rightarrow A_{2n}^{(2)}\\
        x^\pm_{i,r}&\mapsto x^\pm_{i+1,r}\\
        h_{i,r}&\mapsto h_{i+1,r}\\
        c&\mapsto c
    \end{align}
\end{definition}

\begin{definition}\label{bhuz}
Here we define some $\Z$-subalgebras of $\u$:\\
$\uz$ is the $\Z$-subalgebras of $\u
$ generated by \begin{align}
    &\{(x_{i,r}^{+})^{(k)},(x_{i,r}^{-})^{(k)}\mid r\in\Z,k\in\N, i\in I\};
\end{align}\\
$\uz^+$ and $\uz^-$ are the $\Z$-subalgebras of $\u
$ respectively generated by \begin{align}
    &\{(x_{i,r}^{+})^{(k)}\mid r\in\Z,k\in\N, i\in I\},\\ &\{(x_{i,r}^{-})^{(k)}\mid r\in\Z,k\in\N, i\in I\};
\end{align}\\
$\buz^+$ and $\buz^-$ are the $\Z$-subalgebras of $\u
$ respectively generated by \begin{align}
    &\{(x_{i,r}^{+})^{(k)}, (\frac{1}{2}X_{1,2r+1}^{+})^{(k)}\mid r\in\Z,k\in\N, i\in I\},\\ &\{(x_{i,r}^{-})^{(k)}, (\frac{1}{2}X_{1,2r+1}^{-})^{(k)}\mid r\in\Z,k\in\N, i\in I\};
\end{align}\\
Let $n=1$ and  $\epsilon(r)=-1$ if $r\vert 4$ and $\epsilon(r)=1$ otherwise, we define $\uz^{0,+}$ and $\uz^{0,-}$ as the $\Z$-subalgebras of $\u
$ respectively generated by \begin{align}
    &\{\tilde h_{1,r}\mid r>0\},\\ &\{\tilde h_{1,r},\mid r<0\},
\end{align}\\
$\buz^{0,+}$ and $\buz^{0,-}$ are the $\Z$-subalgebras of $\bu
$ respectively generated by \begin{align}
    &\{\hat h_{i,r},\bar h_{1,r},\mid r>0, i\in I\},\{\hat h_{i,r},\bar h_{1,r},\mid r<0, i\in I\},
\end{align}\\
$\muz^{0,+}$ and $\muz^{0,-}$ are the $\Z$-subalgebras of $\u
$ respectively generated by \begin{align}
    &\{\check h_{1,r}\mid r>0, i\in I\},\{ \check h_{1,r},\mid r<0, i\in I\}.
\end{align}\\
$\uz^{\h}=\Z^{bin}[h_{i,0},c\mid i\in I]$;\\  
$\uz^{0}$ is the $\Z$ subalgebra of $\u^0$ generated by $\uz^{0,+},\uz^{0,-}$ and $\uz^{\h}$;\\
$\muz^{\h}=\Z^{bin}[h_{i,0},\frac{c}{2}\mid i\in I]$;\\
$\muz^{0}$ is the $\Z$ subalgebra of $\muz^0$ generated by $\muz^{0,+}$, $\muz^{0,-}$ and $\muz^{\h}$;\\
$\buz^{0}$ is the $\Z$ subalgebra of $\buz^0$ generated by $\buz^{0,+}$, $\buz^{0,-}$ and $\uz^{\h}$.\\
\end{definition}

%\begin{remark}    Of course we have that  $\buz^{0,\pm}\subseteq\Z[\check h_{1,r}, \hat h_{i,r}\mid i\in I\setminus\{1\}, \pm r>0]$.  $\buz^{0,\pm}$ and $\buz^{\h}$ are integral form of respectively $\bu^{0,\pm}$ and $\bu^{\h}$.\end{remark}

\section{Integral forms of $A_{2}^{(2)}$}\label{theothers}
In this Section we fix $X_{\tilde n}^{(k)}=A_{2}^{(2)}$ and we denote by $\u$ its universal enveloping algebra. \\
In this case, we have $I=\{1\}$, and since this does not cause any confusion, we omit the subscript  $i$, e.g. we will denote $x_{i,r}^+$ as $x_{r}^+$.\\
In \cite{DP} we described the integral form  $\uz$ (that we denoted by $\tilde\uz$) of the enveloping algebra  $\u$ (see Definition \ref{a22}) of the Kac-Moody algebra of type $A_2^{(2)}$ generated by the divided powers of the Drinfeld generators $x_{r}^{\pm}$, we recall the result in Theorem \ref{trmA22}.
In this section we describe two other integral forms $\buz$ and $\muz$  (see Definition \ref{bhuz}), $\buz$ is generated by the divided powers of the Drinfeld generators $x^\pm_{r}$ and by the divide powers of the elements $\frac{1}{2}X^\pm_{2r+1}$, $\muz$ is generated by adding extra elements  $\check h_{r}$ to $\buz$ (see Definition \ref{barradef}). As we shall see later (see Section \ref{chapteraqd}), if we consider the $\Z$-algebra generated by the divided powers of the positive Drinfeld  generators $x^+_{i,r}$ ($i\in I $, $r\in\Z$) in the case of $A_{2n}^{(n)}$ for $n>1$ then this algebra also contains the divided powers of the elements  $\frac{1}{2}X_{1,2r+1}^+$, for this reason we are interested in the study of $\buz$. %Following the mantra of previous chapters, the  relations needed to our aim can be  translated directly from $\tuz$ in $\buz$ (see Lemma \ref{nuoveadd} and Proposition \ref{buzaccazero}). 
There are two remarkable differences between $\buz$ and $\uz$: the first, as previously announced, is the presence of the divided powers of the elements $\frac{1}{2}X^\pm_{2r+1}$. The second  difference concerns the structure of the (positive and negative) imaginary component. In fact, in this case $\buz\cap \u^{0,+}\neq \Z[\tilde h_{r}\mid r>0]$  
is no longer an algebra of polynomials (see Proposition \ref{notpol} and Theorem \ref{lambdabaseth} and \ref{polbasis}), but we exhibit a Garland-type $\Z$ basis (see the description if $\Z^{(mix)}[h_r\mid r>0]$ in Definition \ref{mixdef}). We shall also show that $\buz$ can be enlarged to another integral form $\muz$ of $\tu$ with the same positive real part (that is $\muz\cap\tu^+=\buz\cap\tu^+=\buz^+)$ %is the $\Z$-subalgebra of $\tu$ generated by $\{(x^+_r)^{(k)}, (X_{2r+1}^+)^{(k)}\mid r\in \Z, k\in \N\}$) and such that 
but such that 
$\muz\cap\tu^{0,+}=\muz^{0,+}\supsetneq\buz^{0,+}$ is an algebra of polynomials. $\muz$ and $\buz$ will be introduced and studied together and the description of  $\muz$ will also avoid unnecessary computation in $\buz$.
%This Section is organized as follows:

\begin{definition} \label{a22}
 $A_{2}^{(2)}$ (respectively $\tu$) is the Lie algebra (respectively the associative algebra) over $\Q$ generated by $\{c,h_r,x_r^{\pm},X_{2r+1}^{\pm}\mid r\in\Z\}$ with relations $$c\,\,\,{\rm{is\,\,central}}\,,$$
$$[h_r,h_s]=\delta_{r+s,0}2r(2+(-1)^{r-1})c\,;$$
$$[h_r,x_s^{\pm}]=\pm 2(2+(-1)^{r-1})x_{r+s}^{\pm}\,;$$
$$[h_r,X_s^{\pm}]=\begin{cases}\pm 4X_{r+s}^{\pm}&{\rm{if}}\ 2\mid r\,;
\\ 0&{\rm{if}}\ 2\nmid r\,;
\end{cases}\leqno{(s\ {\rm{odd}})}$$
$$[x_r^{\pm},x_s^{\pm}]=\begin{cases}0&{\rm{if}}\ 2\mid r+s\,,
\\ \pm(-1)^sX_{r+s}^{\pm}&{\rm{if}}\ 2\nmid r+s\,;
\end{cases}$$
$$[x_r^{\pm},X_s^{\pm}]=[X_r^{\pm},X_s^{\pm}]=0\,;$$
$$[x_r^+,x_s^-]=h_{r+s}
+\delta_{r+s,0}rc\,;$$
$$[x_r^{\pm},X_s^{\mp}]=\pm(-1)^r4x_{r+s}^{\mp}\,;\leqno{(s\ {\rm{odd}})}$$
$$[X_r^+,X_s^-]=8h_{r+s}
+4\delta_{r+s,0}rc\,.\leqno{(r,s\ {\rm{odd}})}$$
\end{definition}

\begin{remark}\label{tsb}
 $\sigma\vert_{\tu^{\pm,0}}={\rm{id}}_{\tu^{\pm,0}},\,\sigma\vert_{\tu^{\pm,1}}={\rm{id}}_{\tu^{\pm,1}},\,\sigma(\tu^{\pm,c})=\tu^{\pm,c},\,\sigma(\tu^{0,\pm})=\tu^{0,\pm}$, $\sigma(\tu^\h)=\tu^\h$.

 $\Omega(\tu^{\pm,0})\!=\tu^{\mp,0}$, $\Omega(\tu^{\pm,1})\!=\tu^{\mp,1}$, $\Omega(\tu^{\pm,c})\!=\tu^{\mp,c}$, $\Omega(\tu^{0,\pm})\!=\tu^{0,\mp}$, $\Omega\vert_{\tu^\h}\!\!=\!{\rm{id}}_{\tu^\h}$.
\end{remark}

Here we recall the results on ${A_{2}^{(2)}}$  (see \cite{DP}, Theorem 5.46):

\begin{theorem}\label{trmA22}
The $\Z$-subalgebra $\tuz$ of $\tu$ generated by $$\{(x_r^+)^{(k)},(x_r^-)^{(k)}\mid r\in\Z,k\in\N\}$$ is an integral form of $\tu$. 

More precisely
\begin{align}
    &\tuz\cong\tuz^{-}\otimes\tuz^{0}\otimes\tuz^{+},\\
    &\tuz^{0}\cong \tuz^{0,-}\otimes\tuz^\h\otimes\tuz^{0,+},\\
    &\tuz^{\pm}\cong \tuz^{\pm,0}\otimes\tuz^{\pm,c}\otimes\tuz^{\pm,0};
\end{align}

and a $\Z$-basis of $\tuz$ is given by the product
$$B^{-,1}B^{-,c}B^{-,0}B^{0,-}B^\h B^{0,+}B^{+,1}B^{+,c}B^{+,0}$$
where $B^{\pm,0}$, $B^{\pm,1}$, $B^{\pm,c}$, $B^{0,\pm}$ and $B^\h$ are the $\Z$-bases respectively of $\tuz^{\pm,0}$, $\tuz^{\pm,1}$, $\tuz^{\pm,c}$, $\tuz^{0,\pm}$ and $\tuz^\h$ given as follows:
$$B^{\pm}=\Big\{\prod_{r\in\Z}(x_{r}^{\pm})^{(k_r)}\mid{\bf{k}}:\Z\to\N\,\, {\rm{is\, finitely\, supported}}\Big\}$$
$$B^{\pm,1}=\Big\{{(\bf{x}}^{\pm,1})^{({\bf{k}})}=\prod_{r\in\Z}(x_{2r+1}^{\pm})^{(k_r)}\mid{\bf{k}}:\Z\to\N\,\, {\rm{is\, finitely\, supported}}\Big\}$$
$$B^{\pm,c}=\Big\{{(\bf{X}}^{\pm})^{({\bf{k}})}=\prod_{r\in\Z}(X_{2r+1}^{\pm})^{(k_r)}\mid{\bf{k}}:\Z\to\N\,\, {\rm{is\, finitely\, supported}}\Big\}$$
$$B^{0,\pm}=\Big\{{\tilde{{\bf{h}}}_{\pm}^{\bf{k}}}=\prod_{l\in\N}{\tilde h_{\pm l}^{k_l}}\mid{\bf{k}}:\N\to\N\,\,{\rm{is\, finitely\, supported}}\Big\}$$
$$B^\h=\Big\{\binom{h_0}{k}\binom{c}{\tilde k}\mid k,\tilde k\in\N\Big\}.$$
\end{theorem}

\begin{proposition}\label{stbuzmuz}

 The following stability properties under the action of $\sigma$, $\Omega$ and $T^{\pm 1}$  hold:
\begin{itemize}
    \item $\buz$, $\buz^+$,$\buz^-$
    $\buz^{+,0}$, $\buz^{+,1}$,$\buz^{+,c}$,$\buz^\h$, $\buz^{0,+}$, $\buz^{0,-}$ and $\buz^0$ are $\sigma$-stable.
\item 
$\buz$, $\buz^0$ and $\buz^\h$ are $\Omega$-stable 
, while \begin{align}
&\Omega(\buz^{\pm})=\buz^{\mp},\\
&\Omega(\buz^{+,0})=\buz^{-,0},\\ 
&\Omega(\buz^{+,1})=\buz^{-,1},\\ 
&\Omega(\buz^{+,c})=\buz^{-,c},\\ 
    \end{align}
    \item $\tuz$, $\tuz^+$, $\tuz^-$,$\tuz^{+,c}$,
 $\tuz^\h$, $\tuz^{0,+}$,  $\tuz^{0,-}$ and  $\tuz^0$ are
 $T^{\pm 1}$-stable, while 
 $T^{\pm 1}(\tuz^{+,0})=\tuz^{+,1}$ and hence are $T^{\pm 2}$-stable .
\end{itemize}
\end{proposition}
\begin{proof}
The Proof is the same as \cite{DP} Proposition  5.19.
\end{proof}

\begin{remark}
We have that:
\begin{align}
\tuz\subsetneq\buz\subseteq\muz,
\end{align} the first inclusion follows from Theorem \ref{trmA22} and Definition \ref{aquattrodef}, the second inclusion follows from Definition \ref{aquattrodef}.
\end{remark}

The aim of this section is to prove the following theorems:

\begin{theorem}\label{buztheorem}
The $\Z$-subalgebra $\buz$ of $\tu$ is an integral form of $\tu$.
    More precisely
    \begin{align} \buz&\cong\buz^-\otimes\buz^{0,-}\otimes\buz^\h\otimes\buz^{0,+}\otimes\buz^{+}, \\
    \buz^{\pm}&\cong \Z^{(div)}[x^\pm_{2r}\mid \pm r\ge0]\otimes\Z^{(div)}[\frac{1}{2}X^\pm_{2r+1}\mid \pm r\ge0] \otimes\Z^{(div)}[x^\pm_{2r+1}\mid \pm r\ge0],\\
    \buz^0&\cong\buz^{0,-}\otimes\buz^{\h}\otimes\buz^{0,+},
    \end{align}
where $\buz^{\pm}$ has basis 
$B^\pm$ given by the product $B^\pm=B^{\pm,1}B^{\pm,c}B^{\pm,0} $ defined as follows:
$$B^{\pm,1}=\Big\{{(\bf{x}}^{\pm,1})^{({\bf{k}})}=\prod_{r\in\Z}(x_{2r+1}^{\pm})^{(k_r)}\mid{\bf{k}}:\Z\to\N\,\, {\rm{is\, finitely\, supported}}\Big\}$$
$$B^{\pm,0}=\Big\{{(\bf{x}}^{\pm,1})^{({\bf{k}})}=\prod_{r\in\Z}(x_{2r}^{\pm})^{(k_r)}\mid{\bf{k}}:\Z\to\N\,\, {\rm{is\, finitely\, supported}}\Big\}$$
$$B^{\pm,c}=\Big\{{(\bf{X}}^{\pm})^{({\bf{k}})}=\prod_{r\in\Z}(\frac{X_{2r+1}^{\pm}}{2})^{(k_r)}\mid{\bf{k}}:\Z\to\N\,\, {\rm{is\, finitely\, supported}}\Big\},$$
$\buz^{0,\pm}$ 
with basis given by \begin{align}
         B_{q.pol}^{0,\pm}=\{ \prod_{k>0}\hat h_{\pm k}^{\epsilon_k}\prod_{k>0}\bar h_{\pm k}^{d_k}\mid \epsilon,d:\Z_{>0}\rightarrow \N \text{ and } \text{ are finitely supported and } \epsilon_i\in\{0,1\} \} 
    \end{align} or equivalently
    \begin{align}
        B_{q.\lambda}^{0,\pm}=\{ \prod_{m>0, m \text{ odd } } \lambda_m (\hat h_{k_m})\prod_{m>0, m \text{ even } } \lambda_m (\check h_{k_m}),\mid k:\Z_{>0}\rightarrow\N \text{ is finitely supported}\};\end{align}
$\buz^\h=\Z^{(bin)}[h_0,c]$ with basis  \begin{align}
    \bar B^\h=\{\binom{h_0}{r},\binom{c}{s}\mid r,s\in \N\}.
\end{align}
\end{theorem}

\begin{theorem}\label{muztheorem}
The $\Z$-subalgebra $\muz$ of $\tu$ generated by \begin{align}
        \{(\frac{1}{2}X^+_{2r+1})^{(k)},(\frac{1}{2}X^-_{2r+1})^{(k)}, (x^+_{r})^{(k)},(x^-_{r})^{(k)},\check h_{s}\mid r \in \Z, s \in \Z^{*}, k\in \N\}
    \end{align}  is an integral form of $\tu$.
    More precisely
    \begin{align} \muz&\cong\buz^-\otimes\muz^{0,+}\otimes\muz^\h\otimes\muz^{0,+}\otimes\buz^{+},
    \end{align} where $\buz^- $ and $\buz^{+}$ are as in Theorem \ref{buztheorem}, $\muz^{0,\pm}=\Z[\check h_r\mid \pm r>0]$ with basis
$$\check B^{0,\pm}=\Big\{{{{\check h}}^{\pm\bf{k}}}=\prod_{l\in\N}{\check h_{\pm l}^{k_l}}\mid{\bf{k}}:\N\to\N\,\,{\rm{is\, finitely\, supported}}\Big\},$$
$\muz^\h=\Z^{(bin)}[h_0,\frac{c}{4}]$ has basis
$$B^\h=\Big\{
\binom{h_0}{k}
\binom{\frac{c}{4}}{\tilde k}
\mid k,\tilde k\in\N\Big\}.$$
\end{theorem}

\begin{remark}
$\buz^{0,\pm}=\Z^{(mix)}[h_r\mid \pm r>0]$.    
\end{remark}

\begin{remark}\label{leacca} 
The following identities holds in $\u[[u,v]]$ from Definition \ref{barradef} and Remark \ref{rembarceckbo}: \begin{align}     &\label{eq:hateta}\hat h^\pm(u)=\check h^\pm(u)^2,\\    &\label{eq:bareta}\bar h^\pm(u^2)=\check h^\pm(u)\check h^\pm(-u),\\     
&\tilde h^\pm(u)=\hat h^\pm(u)\lambda_{4} (\hat h^\pm(-u^4)^{-\frac{1}{2}})=\hat h^\pm(u)\lambda_2((\bar h^\pm(u^2)^{-1}))
%=\check h^\pm(u)^{2}\lambda_{4} (\check h^\pm(-u^4)^{-1})
.
    \end{align}
\end{remark}

\begin{proposition}\label{muzubuzubo}
The following identities holds in $\u[[u,v]]$:
\begin{align}\label{eq:etaeta}
&\check h^+(u)\check h^-(u)=\check h^-(u)(1-uv)^{c}(1+uv)^{-\frac{c}{2}}\check h^+(u);\\
&\hat h^+(u)\hat h^-(u)=\hat h^-(u)(1-uv)^{2c}(1+uv)^{-{c}}\hat h^+(u);\label{eq:capcap}\\
&\bar h^+(u^2)\bar h^-(v^2)=\bar h^-(v^2)(1-(uv)^2)^{2c}(1-(uv)^2)^{-c}\bar h^+(u^2);\label{eq:barbar}\\
&\hat h^+(u)\bar h^-(v^2)=\bar h^-(v^2)(1-(uv)^2)^c\hat h^+(u)\label{eq:capbar}.
\end{align}
In particular $\muz^0\cong \muz^{0,-}\otimes\muz^\h\otimes\muz^{0,+}\subseteq$ and 
$\buz^0\cong \buz^{0,-}\otimes\buz^\h\otimes\buz^{0,+}$, thus $\muz^{0}$ and $\buz^0$ are an integral forms of $\tu^0$.
\end{proposition}
\begin{proof}
Equation \eqref{eq:etaeta} follows from \cite{DP},Proposition 2.4  with $m=1$, $l=\frac{1}{2}$ by substituting $\frac{c}{2}$ in place of $c$ observing that 
$[\frac{1}{2}h_r,\frac{1}{2}h_s]=\delta_{r+s,0}r(2+(-1)^{r-1})\frac{1}{2}c$.
    Equations \eqref{eq:capcap}, \eqref{eq:barbar} and \eqref{eq:capbar}  follow from Equation \eqref{eq:etaeta} and Remark \ref{leacca}.
\end{proof}

\begin{remark}\label{remtildecheck}
It is worth underling that \begin{align}
&\uz^{0,\pm}\subsetneq\buz^{0,\pm}\subsetneq\muz^{0,\pm}\quad \tuz^{\h}\subsetneq\muz^{\h}.\label{eq:tuzmuzh}
    \end{align}
The first inclusion follows from \cite{DP},Remark 5.13 and Remark \ref{leacca}; the second inclusion  follows from Lemma \ref{chechdiverso}; the third is obvious, e.g. $\frac{c}{2}\not\in \Z^{(bin)}[h_0,c]$.
\end{remark}

\begin{lemma}\label{toglitilde}
 \begin{align}\label{eq:toglitilde}
 \Z(\tilde h_r,\hat h_r^{\{c\}}, \bar h_{2r}\mid \pm r>0 )= \Z^{(mix)}[h_r\mid \pm r>0],
\end{align} 
in particular $\tuz^{0,-}\tuz^{0,+}\subseteq\tuz^0$ and 
    $\tuz^0$ is an integral form of $\tu^0$.
\begin{proof}
    
Let us observe  that
 thus \begin{align}
        \Z( \bar h_{r},\hat h^{\{c\}}, \tilde h_r\mid r>0)=\Z( \bar h_{r},\hat h^{\{c\}}, \tilde h_r\mid r>0),
    \end{align}
moreover Theorem \ref{toglicerchietti}  it follows that $\Z(\hat h_r, \hat h_r^{\{c\}}, \bar h_{2r}, \mid \pm r>0 )=\Z^{(mix)}[h_r\mid\pm r>0]$.
The last assertion follows from Relation \eqref{eq:toglitilde} and Proposition \ref{muzubuzubo}.
\end{proof}
\end{lemma}

\begin{lemma}\label{zdiv}
$\buz^{\pm}$ is an integral form of $\u^+$, more precisely \begin{align}
    \buz^{\pm}&\cong \Z^{(div)}[x^\pm_{2r}\mid r\in\Z]\otimes\Z^{(div)}[\frac{1}{2}X^\pm_{2r+1}\mid r\in\Z] \otimes\Z^{(div)}[x^\pm_{2r+1}\mid \pm r\in\Z]
\end{align} and $\uz^{\pm}\subsetneq\buz^{\pm}$
\end{lemma}\begin{proof}
The claim follows from Theorem \ref{trmA22} observing that $X^\pm_{2r+1}$ is central in $\tuz^\pm$.
\end{proof}

\begin{proposition} \label{buzaccazero}
The following identity holds in $\tu$ $\forall k\in \N$ and $\forall r\in \Z$:
\begin{align}
    &(\frac{1}{2}X_{2r+1}^\pm)^{(k)}\binom{h_0}{l}=\binom{h_0\mp4k}{ l}(\frac{1}{2}X_{2r+1}^\pm)^{(k)}\label{eq:xgrandicartanmezzi},
\end{align}
in particular $\buz^{\pm}\tuz^\h=\tuz^\h\buz^{\pm}$.
\end{proposition}
\begin{proof}
Equation \eqref{eq:xgrandicartanmezzi} follows from \cite{DP}, Appendix A, V) by multiplying both side by $(\frac{1}{2})^k$.
The claim follows by \cite{DP},Proposition 5.24 and Equation \eqref{eq:xgrandicartanmezzi}. The last equality follows from Lemma \ref{zdiv} and \cite{DP}, Proposition 5.24.
\end{proof}

\begin{proposition}\label{zeropiubarra}
The following relations hold in $\tu[[u]]$
\begin{align}
    &x_{0}^+\check h^+(u)=\check h^+(u)(1-T^{-1}u)(1-T^{-2}u^2)^{-3}(x_{0}^+)\label{eq:commupieta}\\&X_{1}^+\check h^+(u)=\check h^+(u)(1-T^{-1}u^2)^{-1}(X_{1}^+)\label{eq:commupietagrande}\\
    \end{align}
    hence for all $k\ge 0$
    \begin{align}
    &\big(x_{0}^+\big)^{(k)}\check h^+(u)=\check h^+(u)\big((1-T^{-1}u)(1-T^{-2}u^2)^{-3}(x_{0}^+)\big)^{(k)}\in \muz^{0,+}\buz^+[[u]] \label{eq:kuno}\\&\big(\frac{1}{2}X_{1}^+\big)^{(k)}\check h^+(u)=\check h^+(u)\big((1-T^{-1}u^2)^{-1}\frac{1}{2}(X_{1}^+)\big)^{(k)}\in \muz^{0,+}\buz^+[[u]]
    \label{eq:kunoo}\\
    \\\end{align} and
        \begin{align}
&\big(x_{0}^+\big)^{(k)}\hat h^+(u)=\hat h^+(u)\big((1-T^{-1}u)^{-2}(1-T^{-2}u^2)^{-6}(x_{0}^+)\big)^{(k)}
\in \buz^{0,+}\buz^+[[u]], \label{eq:huno}\\&
\big(\frac{1}{2}X_{1}^+\big)^{(k)}\hat h^+(u)=\hat h^+(u)\big((1-T^{-1}u^2)^{-2}\frac{1}{2}(X_{1}^+)\big)^{(k)}\in \buz^{0,+}\buz^+[[u]],\label{eq:hunoo}\\
&\big(x_{0}^+\big)^{(k)}\bar h^+(u)=
\bar h^+(u)\big(
(1-T^{-2}u^2)^{-5}
(x_{0}^+)\big)^{(k)}
\in \buz^{0,+}\buz^+[[u]], \label{eq:buno}\\
&\big(\frac{1}{2}X_{1}^+\big)^{(k)}\bar h^+(u)=\bar h^+(u)\big((1-T^{-1}u^2)^{-2}\frac{1}{2}(X_{1}^+)\big)^{(k)}\in \muz^{0,+}\buz^+[[u]].
    \label{eq:bunoo}
\end{align}
In particular 
\begin{align}
    \buz^+\muz^{0,\pm}\subseteq \muz^{0,\pm} \buz^+,\label{eq: buzpiumuzzeropiu}\\
        \buz^-\muz^{0,\pm}\subseteq \muz^{0,\pm} \buz^-,\label{eq: buzpiumuzzeropiudue}\\
        \buz^+\buz^{0,\pm}\subseteq \buz^{0,\pm} \buz^+,\label{eq: altrauno}\\
        \buz^-\buz^{0,\pm}\subseteq \buz^{0,\pm} \buz^-\label{eq: altradue}\\
    \end{align}
thus $\muz^0\buz^+$ and $\buz^0\buz^+$  (respectively $\buz^-\muz^0$ and $\buz^-\buz^0$) are 
 integral form of $\tu^0\tu^+$ (respectively of $\tu^-\tu^0$).
    \end{proposition}
    \begin{proof}
        Equations \eqref{eq:commupieta} and \eqref{eq:commupietagrande} follow from
\cite{DP},Proposition 2.14 respectively with $m_1=1$, $m_2=3$ and $m_d=0$ if $d>2$ and $m_2=1$ and $m_d=0$ if $d>2$. Equations \eqref{eq:kuno} and \eqref{eq:kunoo} follow respectively by Equation \eqref{eq:commupieta} and \eqref{eq:commupietagrande}. 
From the $T^{\pm}$ stability of $\buz^+$ and the fact that $T\vert_{\tu^{0,+}}=\text{id}\vert_{\tu^{0,+}}$ we deduce that  for all $k\ge 0$ $(x_{r}^{+})^{(k)}\check h^+(u)\subseteq \check h^+(u) \buz^+[[u]]$ and $(\frac{1}{2}X_{2r+1}^{+})^{(k)}\check h^+(u)\subseteq \check h^+(u) \buz^+[[u]]$. Recalling that the $\check h_r$ generate $\muz^{0,+}$ and the $(x^+_{r})^{(k)}$ and the $(\frac{1}{2}X^+_{r})^{(k)}$ generate $\buz^{+}$ follows that $\buz^+\muz^{0,+}\subseteq \muz^{0,+} \buz^+$, \eqref{eq: buzpiumuzzeropiudue} that is Relation \eqref{eq: buzpiumuzzeropiu} follows applying $\Omega$. Equations \eqref{eq:huno} and \eqref{eq:hunoo} follows from Equations \eqref{eq:kuno} and \eqref{eq:kunoo} remembering that $\hat h^+(u)=\check h^+(u)^{2}$,
Equations \eqref{eq:buno} and \eqref{eq:bunoo} follows from Equations \eqref{eq:kuno} and \eqref{eq:kunoo} remembering that $\bar h^+(u)=\check h^+(u)^{2}\check h^+(-u)^{2}$
.
Relation \eqref{eq: altrauno} follows from Equations \eqref{eq:huno}, \eqref{eq:hunoo}, \eqref{eq:kuno}and  \eqref{eq:kunoo}.
Relation \eqref{eq: altradue}
and applying $\Omega$. The last equality follows from Proposition \ref{buzaccazero}.
    \end{proof} 

\begin{lemma} \label{nuoveadd}
      The following identities hold in $\tu[[u,v]]$: 
 
\begin{align}\label{eq:identuno}
&\exp(x_{0}^+u)\exp(\frac{1}{2}X_{1}^-v)=\\
&\exp\left(
\frac{2}{1-4T^2u^4v^2}
x_{0}^-uv\right)\exp\left(
\frac{-4T^2}{1-4T^2u^4v^2}
x_{1}^-u^3v^2\right)\cdot\\
&\cdot\exp\left({1-3\cdot 4Tu^4v^2\over (1-4T^{1}u^4v^2)^2}\frac{1}{2}X_{1}^-v\right)
\hat h^+(2u^2v)^{\frac{1}{2}}
\exp\left({1+4T^{-1}u^4v^2\over (1-4T^{-1}u^4v^2)^2}\frac{1}{2}X_{1}^+u^4v\right)\cdot\\&\cdot\exp\left({-2\over 1-4T^{-2}u^4v^2}x_1^+u^3v\right)\exp\left(
\frac{1}{1-4T^{-2}u^4v^2}
x_{1}^+u\right);\\
%\end{align}\begin{align}\\
\\
\label{eq:identdue}
&\exp(\frac{1}{2}X_{2r+1}^+u)\exp(\frac{1}{2}X_{2s-1}^-v)=\\&
    \exp\left(
    \frac{1}{1+T^{s+r}uv}
    \frac{1}{2}X_{2s-1}^-v\right)\cdot
\lambda_{2(r+s)}(\hat h^+((u^rv^s)^{2})^{
\frac{1}{2}})
    \cdot\exp\left(
    \frac{1}{1+uvT^{-s-r}}
    \frac{1}{2}X_{2r+1}^+u\right), \text{ if } r+s\neq0;\\
&
    \exp\left(
    \frac{1}{1+T^{s+r}uv}
    \frac{1}{2}X_{2s-1}^-v\right)\cdot
\lambda_{r+s}(\bar h^+(-u^rv^s))
    \cdot\exp\left(
    \frac{1}{1+uvT^{-s-r}}
    \frac{1}{2}X_{2r+1}^+u\right), \text{ if } r+s\neq0;\\
    \label{eq:identtre}
&\exp(\frac{1}{2}X_{2r+1}^+u)\exp(\frac{1}{2}X_{2s-1}^-v)=\\&
    \exp\left(\frac{1}{2}X_{2s-1}^-v\right)\cdot
(1+4uv)^{(
\frac{h_0}{2}+
\frac{(2r+1)c}{4})}
    \cdot\exp\left(\frac{1}{2}X_{2r+1}^+u\right), \text{ if } r+s=0;\\
\end{align}
\end{lemma}
\begin{proof}
    Equations \eqref{eq:identuno} 
    follows from \cite{DP},Appendix A, VII,c)   substituting $\frac{1}{2}v$ to $v$.
    Equation \eqref{eq:identdue}  follows from  \cite{DP},Appendix A, VII,b)
     substituting respectively $\frac{1}{2}u$ to $u$ and $\frac{1}{2}v$ to $v$. Equation \eqref{eq:identtre} follows by   \cite{DP},Appendix A,VII,a) substituting $\frac{1}{2}u$ and $\frac{1}{2}v$ respectively to $u$ and $v$ .
\end{proof}

\begin{corollary}\label{ricondurre}
\begin{align}
&\Z^{(div)}[\frac{1}{2}X_{2r+1}^+\mid r>0]\Z^{(div)}[\frac{1}{2}X_{2r+1}^-\mid r>0]\subseteq \buz^{-}\otimes\buz^0 \otimes
\buz^{+},\label{eq:EUA}\\
&\Z^{(div)}[x_{2r}^+\mid r>0]\Z^{(div)}[\frac{1}{2}X_{2r+1}^-\mid r>0]\subseteq\buz^{-}\otimes\buz^0\otimes\buz^{+},\label{eq:EUB}\\&
\bar h_{2r}, \hat h_r^{\{c\}}\in \uz
\label{eq:EUC},\\& \buz^+\buz^-\subseteq \buz^-\otimes\buz^0\otimes\buz^+.\label{eq:EUD}\\&
\muz^+\muz^-\subseteq \muz^-\otimes\muz^0\otimes\muz^+.\label{eq:EUE}
\end{align}
\end{corollary}
\begin{proof} Relations \eqref{eq:EUA},\eqref{eq:EUB},\eqref{eq:EUC} follow from Propositions \ref{zeropiubarra} and \ref{nuoveadd}. Relation \eqref{eq:EUD} follows from Theorem \ref{trmA22}, Relation \eqref{eq:EUC} and Lemma \ref{toglitilde}. Relation \eqref{eq:EUE} follows from Relation \eqref{eq:EUD}, Remark \ref{remtildecheck} and Proposition \ref{zeropiubarra}.
\end{proof}
\begin{remark}
    The last two assertions of Corollary \ref{ricondurre} are the Claim of Theorems \ref{buztheorem} and  \ref{muztheorem}.
\end{remark}

%\begin{proof} From Propositions \ref{zeropiubarra} and \ref{buzaccazero} follow that $\buz^{+}\muz^0\subseteq\muz^0 \buz^{+}$, then the proof is the same as \cite{DP},Corollary 5.26.\end{proof}

\section{ Integral form of $A_{2n}^{(2)}$}\label{chapteraqd}

\subsection{Integral form of $A_4^{(2)}$} \label{glue}
Let us fix in this subsection $n=2$.
In this part we want to study the algebra $\uz^+(A_4^{(2)})$.
The first part is devoted to the study of the positive real roots part, more specifically in the initial part (Lemmas \ref{techuno} and \ref{lemmatauuno} ) 
we will study certain commutation formulas of the Lie algebra, in Lemma \ref{commuplus} we will use the results to study the commutation rules within the enveloping algebra of divided powers, the results obtained will be collected within  Theorem \ref{buzpiudue}.

The finite Lie algebra associate to $A_{4}^{(2)}$ is the Lie algebra of type $B_2$,
let $W_0$ its Weyl group and let $w_0$ its longest element we fix  the following reduced expression for $w_0$: $\sigma_2\sigma_1\sigma_2\sigma_1$.
In particular we the root vectors are the following:
\begin{align}
\{x^+_{2,r},\tau_2(x^+_{1,r})=x^+_{\alpha_1+\alpha_2,r},
\tau_1(x^+_{2,r})=x^+_{2\alpha_1+\alpha_2,r},x^+_{1,r}\}
\end{align}

\begin{lemma}\label{techuno}
    \begin{align}
        &[x^-_{1,0},[x^+_{1,0},x^+_{2,r}]]=2x^+_{2,r}\label{eq:formulaA};\\
        &[x^-_{1,0},[x^+_{1,0},[x^+_{1,0},x^+_{2,r}]]]=2[x_{1,0}^+,x^+_{2,r}]\label{eq:formulaB};\\
        &[x^-_{1,0},[x^-_{1,0},[x^+_{1,0},[x^+_{1,0},x^+_{2,r}]]]]=4x^+_{2,r}\label{eq:formulaC};\\
        &[x^-_{2,0},[x^+_{2,0},x^+_{1,r}]]=x^+_{1,r};\label{eq:formulaD}\\
        &[h_{2,0},X^+_{1,r}]=-2X^+_{1,r}\label{eq:formulaE}\\
        &[x_{2,0}^-,[x_{2,0}^+,X^+_{1,r}]]
        =2X^+_{1,r}\label{eq:formulaF}.\\
        &[x_{2,0}^-,[x_{2,0}^+,[x_{2,0}^+,X^+_{1,r}]]]        =2[x_{2,0}^+,X^+_{1,r}]\label{eq:formulaG}\\
        &[x_{2,0}^-,[x_{2,0}^-,[x_{2,0}^+,[x_{2,0}^+,X^+_{1,r}]]]]=4X^+_{1,r}.     \label{eq:formulaH}
    \end{align}
\end{lemma}
\begin{proof}
Proof of equations  \eqref{eq:formulaA},\eqref{eq:formulaB},\eqref{eq:formulaC} and \eqref{eq:formulaD}.
    \begin{align}
        &[x^-_{1,0},[x^+_{1,0},x^+_{2,r}]]=
        -[x^+_{2,r},[x^-_{1,0},x^+_{1,0}]]=
        [x^+_{2,r},h_{1,0}]=-[h_{1,0},x^+_{2,r}]=2x^+_{2,r}.\\\\
    &[x^-_{1,0},[x^+_{1,0},[x^+_{1,0},x^+_{2,r}]]]=
    -([[x^+_{1,0},x^+_{2,r}],[x^-_{1,0},x^+_{1,0}]]+[x^+_{1,0},[[x^+_{1,0},x^+_{2,r}],x^-_{1,0}]])=
    \\&[[x^+_{1,0},x^+_{2,r}],h_{1,0}]+[x^+_{1,0},[x^-_{1,0},[x^+_{1,0},x^+_{2,r}]]]=-[h_{1,0},[x^+_{1,0},x^+_{2,r}]]+2[x_{1,0}^+,x^+_{2,r}]=\\&[x^+_{2,r},[h_{1,0},x^+_{1,0}]]+[x^+_{1,0},[x^+_{2,r},h_{1,0}]]+2[x_{1,0}^+,x^+_{2,r}]=
    \\&2[x^+_{2,r},x^+_{1,0}]+2[x^+_{1,0},x^+_{2,r}]+2[x_{1,0}^+,x^+_{2,r}]=2[x_{1,0}^+,x^+_{2,r}].\\\\
    &[x^-_{1,0},[x^-_{1,0},[x^+_{1,0},[x^+_{1,0},x^+_{2,r}]]]]=2[x^-_{1,0},[x_{1,0}^+,x^+_{2,r}]]=4x^+_{2,r};\\
\\
    &[x^-_{2,0},[x^+_{2,0},x^+_{1,r}]]=-[x^+_{1,r},[x^-_{2,0},x^+_{2,0}]=
        [x^+_{1,r},h_{2,0}]=x^+_{1,r}.
\end{align}
Proof of equations \eqref{eq:formulaE}, \eqref{eq:formulaF}, \eqref{eq:formulaG} and \eqref{eq:formulaH} 
\begin{align}
    &[h_{2,0},X^+_{1,r}]=[h_{2,0},[x^+_{1,r},x^+_{1,0}]]=\\
&-[x^+_{1,0},[h_{2,0},x^+_{1,r}]]-[x^+_{1,r},[x^+_{1,0},h_{2,0}]]=-[x^+_{1,0},x^+_{1,r}]+[x^+_{1,r},x^+_{1,0}]=-2X^+_{1,r};
\\\\
    &[x_{2,0}^-,[x_{2,0}^+,X^+_{1,r}]]=-[X^+_{1,r},[x_{2,0}^-,x_{2,0}^+]]=[X^+_{1,r},h_{2,0}]=2X^+_{1,r};
\\\\
    &[x_{2,0}^-,[x_{2,0}^+,[x_{2,0}^+,X^+_{1,r}]]]=-[[x_{2,0}^+,X^+_{1,r}],[x_{2,0}^-,x_{2,0}^+]]-[x_{2,0}^+,[[x_{2,0}^+,X^+_{1,r}],x_{2,0}^-]]=\\
    &[[x_{2,0}^+,X^+_{1,r}],h_{2,0}]]+[x_{2,0}^+,[x_{2,0}^-,[x_{2,0}^+,X^+_{1,r}]]]=-[h_{2,0},[x_{2,0}^+,X^+_{1,r}]]
    +2[x_{2,0}^+,X^+_{1,r}]=\\
    &[X^+_{1,r},[h_{2,0},x_{2,0}^+]]+[x_{2,0}^+,[X^+_{1,r},h_{2,0}]]
    +2[x_{2,0}^+,X^+_{1,r}]=\\
    &2[X^+_{1,r},x_{2,0}^+]
    +2[x_{2,0}^+,X^+_{1,r}]
    +2[x_{2,0}^+,X^+_{1,r}]=2[x_{2,0}^+,X^+_{1,r}];\\
\\
     &[x_{2,0}^-,[x_{2,0}^-,[x_{2,0}^+,[x_{2,0}^+,X^+_{1,r}]]]]=2[x_{2,0}^-,[x_{2,0}^+,X^+_{1,r}]]=4X^+_{1,r}.
\end{align}
\end{proof}
\begin{lemma}\label{lemmatauuno}
The following identities hold in $\tu$:
\begin{align}     &\label{eq:tauuxp}\tau_1(x^+_{2,r})=\frac{1}{2}[x_{1,0}^+,[x_{1,0}^+,x_{2,r}^+]]=x_{2\alpha_1+\alpha_2,r}^+;\\
&\tau_2(x^+_{1,r})=[x_{2,0}^+,x_{1,r}^+]=x_{\alpha_1+\alpha_2,r}^+;\label{eq:tauduexpp}\\
        %&\tau_2(x^-_{1,r})=[x_{2,0}^-,x_{1,r}^-];\label{eq:tauduexpm}\\
        &\tau_2(X^+_{1,r})=\frac{1}{4}[x^+_{2,0},[x^+_{2,0},X^+_{1,r}]]=X^+_{2,r}.\label{eq:tauduexgp}
    \end{align}
\end{lemma}
\begin{proof} We use relations of Lemma \ref{techuno}.\\
 Proof of Equation \eqref{eq:tauuxp}:  
    \begin{align}     \tau_1(x^+_{2,r})&=\exp(\text{ad}x^+_{1,0})\exp(-\text{ad}x^-_{1,0})\exp(\text{ad}x^+_{1,0})\big(x_{2,r}^+\big)\\   &=\exp(\text{ad}x^+_{1,0})\exp(-\text{ad}x^-_{1,0})\big(x_{2,r}^++[x_{1,0}^+,x_{2,r}^+]+\frac{1}{2}[x_{1,0}^+,[x_{1,0}^+,x_{2,r}^+]]\big)\\   &=\exp(\text{ad}x^+_{1,0})\big(x_{2,r}^++[x_{1,0}^+,x_{2,r}^+]-[x_{1,0}^-,[x_{1,0}^+,x_{2,r}^+]]\\&\quad+\frac{1}{2}[x_{1,0}^+,[x_{1,0}^+,x_{2,r}^+]]-\frac{1}{2}[x_{1,0}^-,[x_{1,0}^+,[x_{1,0}^+,x_{2,r}^+]]]+\frac{1}{4}[x_{1,0}^-,[x_{1,0}^-,[x_{1,0}^+,[x_{1,0}^+,x_{2,r}^+]]]\big)\\   &=\exp(\text{ad}x^+_{1,0})\big(x_{2,r}^++[x_{1,0}^+,x_{2,r}^+]-2x_{2,r}^++\frac{1}{2}[x_{1,0}^+,[x_{1,0}^+,x_{2,r}^+]]-[x_{1,0}^+,x_{2,r}^+]+x_{2,r}^+\big)\\  &=\exp(\text{ad}x^+_{1,0})(\frac{1}{2}[x_{1,0}^+,[x_{1,0}^+,x_{2,r}^+]])=\frac{1}{2}[x_{1,0}^+,[x_{1,0}^+,x_{2,r}^+]].
    \end{align}

Proof of Equation \eqref{eq:tauduexpp}:
    \begin{align}
\tau_2(x^+_{1,r})&=\exp(\text{ad}x^+_{2,0})\exp(-\text{ad}x^-_{2,0})\exp(\text{ad}x^+_{2,0})\big(x_{1,r}^+\big)\\   
&=\exp(\text{ad}x^+_{2,0})\exp(-\text{ad}x^-_{2,0})\big(x_{1,r}^++[x_{2,0}^+,x_{1,r}^+]\big)\\    
&=\exp(\text{ad}x^+_{2,0})\big(x_{1,r}^++[x_{2,0}^+,x_{1,r}^+]-[x_{2,0}^-,[x_{2,0}^+,x_{1,r}^+]]\big)\\
&=\exp(\text{ad}x^+_{2,0})\big(x_{1,r}^++[x_{2,0}^+,x_{1,r}^+]-x^+_{1,r}\big)\\
&=\exp(\text{ad}x^+_{2,0})([x_{2,0}^+,x_{1,r}^+])=[x_{2,0}^+,x_{1,r}^+].
    \end{align}

%Proof of equation \eqref{eq:taudueh}:\begin{align}&\tau_2(h_{1,r})=\exp(\text{ad}x^+_{2,0})\exp(-\text{ad}x^-_{2,0})\exp(\text{ad}x^+_{2,0})\big(h_{1,r}\big)=\\   &\exp(\text{ad}x^+_{2,0})\exp(-\text{ad}x^-_{2,0})\big(h_{1,r}+[x_{2,0}^+,h_{1,r}]\big)=\\&\exp(\text{ad}x^+_{2,0})\exp(-\text{ad}x^-_{2,0})\big(h_{1,r}+2x_{2,r}^+\big)=\\&\exp(\text{ad}x^+_{2,0})\big(h_{1,r}+2x_{2,r}^-+2x_{2,r}^+-2[x_{2,0}^-,x_{2,r}^+]+2\frac{1}{2}[x_{2,0}^-,[x_{2,0}^-,x_{2,r}^+]]\big)=\\&\exp(\text{ad}x^+_{2,0})\big(h_{1,r}+2x_{2,r}^-+2x_{2,r}^++2h_{2,r}-2x_{2,r}^-\big)=\\&\exp(\text{ad}x^+_{2,0})\big(h_{1,r}+2x_{2,r}^++2h_{2,r}\big)=\\&\exp(\text{ad}x^+_{2,0})\big(h_{1,r}+2x_{2,r}^++2h_{2,r}\big)=\\&h_{1,r}+2x_{2,r}^++2x_{2,r}^++2h_{2,r}+2[x_{2,0}^+,h_{2,r}]=\\&h_{1,r}+2x_{2,r}^++2x_{2,r}^++2h_{2,r}-4x_{2,r}^+=\\&h_{1,r}+2h_{2,r}.\end{align}
Proof of Equation \eqref{eq:tauduexgp}:
\begin{align}
\tau_2(X^+_{1,r})&=\exp(\text{ad}x^+_{2,0})\exp(-\text{ad}x^-_{2,0})\exp(\text{ad}x^+_{2,0})\big(X^+_{1,r}\big)\\
&=\exp(\text{ad}x^+_{2,0})\exp(-\text{ad}x^-_{2,0}))\big(X^+_{1,r}+[x^+_{2,0},X^+_{1,r}]+\frac{1}{2}[x^+_{2,0},[x^+_{2,0},X^+_{1,r}]]\big)\\
&=\exp(\text{ad}x^+_{2,0})(X^+_{1,r}+[x^+_{2,0},X^+_{1,r}]-[x^-_{2,0},[x^+_{2,0},X^+_{1,r}]]\\
&+\frac{1}{2}[x^+_{2,0},[x^+_{2,0},X^+_{1,r}]]-\frac{1}{2}[x^-_{2,0},[x^+_{2,0},[x^+_{2,0},X^+_{1,r}]]]+\frac{1}{4}[x^-_{2,0},[x^-_{2,0},[x^+_{2,0},[x^+_{2,0},X^+_{1,r}]]]])\\
&=\exp(\text{ad}x^+_{2,0})(\frac{1}{4}[x^+_{2,0},[x^+_{2,0},X^+_{1,r}]])=\frac{1}{4}[x^+_{2,0},[x^+_{2,0},X^+_{1,r}]].
\end{align}
%Equations \eqref{eq:tauduexpm} and \eqref{eq:tauduexgm} follow respectively from \eqref{eq:tauduexpp} and \eqref{eq:tauduexgp} applying $\lambda_{-1}\circ\Omega$.
\end{proof}
We will now use the $\tau_i$s to prove straightening formulas of the positive real root vectors.
\begin{lemma}\label{commuplus}
The following identities hold in $\tu^+[[u,v]]$
\begin{align}
&i)\exp(x^+_{1,r}u)\exp(x^+_{2,s}v)=\exp(x^+_{2,s}v)\exp(x^+_{1,r}u)\exp(x^+_{\alpha_1+\alpha_2,r+s}uv)\exp((-1)^{r+1}x^+_{2\alpha_1+\alpha_2,2r+s}u^2v),\label{eq:commu1}\\
&ii)\exp(x^+_{1,r}u)\exp(x^+_{\alpha_1+\alpha_2,s}v)=\exp(x^+_{\alpha_1+\alpha_2,r}v)\exp(2(-1)^{r}x^+_{2\alpha_1+\alpha_2,r+s}uv)\exp(x^+_{1,r}u),\label{eq:commu2}\\
&iii)\exp(\frac{1}{2}X^+_{1,r}u)\exp(x^+_{2,r}v)=\exp(x^+_{2,s}v)\exp(\frac{1}{2}X^+_{1,r}u)\exp(2x^+_{2\alpha_1+\alpha_2,r+s}uv),\label{eq:commu3}\\
&iv)\exp(x^+_{2,r}u)\exp(x^+_{2\alpha_1+\alpha_2,s}v)=\exp(x^+_{2\alpha_1+\alpha_2,r}v)
    \exp(-\frac{1}{2}X^+_{2,r+s}uv)\exp(x^+_{2,s}u),\label{eq:commu4}\text{ if } r+s \text{ is odd}. \\
\end{align}  
\end{lemma}
\begin{proof}
Proof of Equation \eqref{eq:commu1}:\\
From \cite{DP},Lemma 2.3,vi) follows that
    \begin{align}
   &\exp(x^+_{1,r}u)\exp(x^+_{2,s}v)
   =\exp(x^+_{2,s}v)\exp(x^+_{1,r}u+[x^+_{1,r},x^+_{2,s}]uv+\frac{1}{2}[x^+_{1,r},[x^+_{1,r},x^+_{2,s}]u^2v)
   \\& =\exp(x^+_{2,s}v)\exp(x^+_{1,r}u)\exp([x^+_{1,r},x^+_{2,s}]uv)\exp(-\frac{1}{2}[x^+_{1,r},[x^+_{1,r},x^+_{2,s}]]u^2v))\\
    \end{align} where the last equality follows from \cite{DP},Lemma 2.3,viii).\\
    Using Relations \eqref{eq:solorpius} and \eqref{eq:rrscambiasegno} follows that 
    \begin{align}
        [x^+_{1,r},x^+_{2,s}]=-[x^+_{2,0},x^+_{1,r+s}]=-x^+_{\alpha_1+\alpha_2,s+r}
    \end{align}
    and \begin{align}
        -\frac{1}{2}[x^+_{1,r},[x^+_{1,r},x^+_{2,s}]]=(-1)^{r+1}([x^+_{1,0},[x^+_{1,0},x^+_{2,s+2r}])=(-1)^{r+1}x^+_{2\alpha_1+\alpha_2,s+2r}.
    \end{align}
 Proof of Equation \eqref{eq:commu2}:
 from \cite{DP},Lemma 2.3,iv) follows that
    \begin{align}
        &\exp(x^+_{1,r}u)\exp(x^+_{\alpha_1+\alpha_2,s}v)=
        \exp(x^+_{\alpha_1+\alpha_2,r}v)
        \exp([x^+_{1,r},x^+_{\alpha_1+\alpha_2,s}]uv)
        \exp(x^+_{1,r}u).
    \end{align} 
    Using Relations \eqref{eq:solorpius} and \eqref{eq:rrscambiasegno} we have that
    \begin{align}
        &[x^+_{1,r},x^+_{\alpha_1+\alpha_2,s}]=[[x^+_{1,r},[x^+_{2,0},x^+_{1,s}]]=
       -[[x^+_{1,r},[x^+_{1,s},x^+_{2,0}]] \\
       &=-[x^+_{1,r},[x^+_{1,r},x^+_{2,s-r}]]=(-1)^{r}[x^+_{1,0},[x^+_{1,0},x^+_{2,s+r}]]=2(-1)^{r}x^+_{2\alpha_1+\alpha_2,r+s}.\\&
    \end{align}
Proof of Equation \eqref{eq:commu3}, from \cite{DP},Lemma 2.3,vi) we get:
    \begin{align}
        \exp(X^+_{1,r}u)\exp(x^+_{2,s}v)
        =\exp(x^+_{2,s}v)\exp(X^+_{1,r}u)
        \exp([X^+_{1,r},x^+_{2,s}]uv),% \exp(\frac{1}{2}[X^+_{1,r},[X^+_{1,r},x^+_{2,s}]u^2v)
    \end{align}
    the claim follows observing that:
\begin{align}
    [X^+_{1,r},x^+_{2,s}]&=[[x^+_{1,r},x^+_{1,0}],x^+_{2,s}]=-[x^+_{2,s},[x^+_{1,r},x^+_{1,0}]]\\
    &=\big([x^+_{1,0},[x^+_{2,s},x^+_{1,r}]]+[x^+_{1,r},[x^+_{1,0},x^+_{2,s}]]\big)\\
    &=[x^+_{1,0},[x^+_{1,0},x^+_{2,s+r}]]-[x^+_{1,r},[x^+_{1,r},x^+_{2,s-r}]]\\
    &=2x^+_{2\alpha_1+\alpha_2,s+r}+(-1)^{r+1}[x^+_{1,0},[x^+_{1,0},x^+_{2,s+r}]]\\
    &=4x^+_{2\alpha_1+\alpha_2,s+r}.
\end{align}
Proof of Equation \eqref{eq:commu4} from \cite{DP},Lemma 2.3,iv) follows that
\begin{align}
    &\exp(x^+_{2,r}u)\exp(x^+_{2\alpha_1+\alpha_2,s}v)=
    \exp(x^+_{2\alpha_1+\alpha_2,r}v)
    \exp([x^+_{2\alpha_1+\alpha_2,r},x^+_{2,r}]uv)\exp(x^+_{2,s}u)
\end{align}
hence the claim follows observing that:
\begin{align}
[x^+_{2\alpha_1+\alpha_2,r},x^+_{2,s}]&=\frac{1}{2}[[x^+_{1,0},[x^+_{1,0},x^+_{2,r}]],x^+_{2,s}]\\&=-\frac{1}{2}[[x^+_{2,s},x^+_{1,0}],[x^+_{1,0},x^+_{2,r}]]=-\frac{1}{2}X^+_{2,r+s}.
\end{align}
\end{proof}

\begin{corollary}
    $\buz^{\pm}\subseteq \uz$, more precisely :
    \begin{enumerate}
        \item $(x_{\alpha_1+\alpha_2,r}^+)^{(k)}$, $(x_{2\alpha_1+\alpha_2,r}^+)^{(k)}$ and $(\frac{1}{2}X_{2,2r+1}^+)^{(k)}$ belong to the $\Z$-subalgebra of $\u$ generated by the $(x_{i,r}^+)^{(k)}$s, in particular they belong to $\uz\cap \u^+$.
        \item $(\frac{1}{2}X_{1,2r+1}^+)^{(k)}\in \uz\cap \u^+$ even if it does not belong to the $\Z$-subalgebra generated by the $(x_{i,r}^+)^{(k)}$s.
    \end{enumerate}
\end{corollary}
\begin{proof}\begin{enumerate}
    \item 
    From Lemma \ref{commuplus},i)  it follows that \begin{align}
        \exp(x_{\alpha_1+\alpha_2,r}^+uv)\exp(x_{2\alpha_1+\alpha_2,r}^+u^2v)\in \Z((x_{i,r}^+)^{(k)}\mid i\in I, r\in \Z, k\in \N)[[u,v]],
    \end{align} then considering the coefficients of $u^kv^k$ and of $u^{2k}v^k$ we get that \begin{align}
    (x_{\alpha_1+\alpha_2,r}^+)^{(k)},
    (x_{2\alpha_1+\alpha_2,r}^+)^{(k)}\in \Z((x_{i,r}^+)^{(k)}\mid i\in I, r\in \Z, k\in \N),\end{align} then Lemma \ref{commuplus},iv) implies that \begin{align}
        (\frac{1}{2}X_{2,2r+1}^+)^{(k)}\in \Z((x_{i,r}^+)^{(k)}\mid i\in I, r\in \Z, k\in \N);
    \end{align}
    \item $\uz$ is $\tau_2$-invariant, hence \begin{align}
        \u^+\ni(\frac{1}{2}X_{1,2r+1}^+)^{(k)}=\tau_2(\frac{1}{2}X_{2,2r+1}^+)^{(k)}\in\uz,
    \end{align} but $(\frac{1}{2}X_{1,2r+1}^+)^{(k)}\not\in \Z ((x_{i,r}^+)^{(k)}\mid i\in I, r\in \Z, k\in \N)$ (see Section \ref{theothers}).
    \end{enumerate}
\end{proof}
 
\begin{theorem}\label{buzpiudue}
$\buz^+\subseteq\uz\cap\u^+$ and $\buz^-\subseteq\uz\cap\u^-$ are integral form of respectively  $\u^+$ and $\u^-$, a $\Z$-basis of $\uz^\pm$ is given by the ordered monomials of the set:
\begin{align}
    \{(x^\pm_{\alpha,r})^{(k)},(\frac{1}{2}X^{\pm}_{i,2r+1})^{(k)}\mid \alpha\in \Phi^+_{0},i\in I, r\in \Z, k\in \N\}.
\end{align}
\end{theorem}
\begin{proof}
 From Lemma \ref{commuplus} follows that the $\Z$-subalgebra of $\uz$ generated by $\{(x^+_{i,r})^{(k)}\mid i\in I, r\in \Z, k\in \N\}$ has basis consisting in the ordered monomials in the set \begin{align}
        \{(x^+_{\alpha,r})^{(k)},(\frac{1}{2}X^+_{2,2r+1})^{(k)},(X^+_{1,2r+1})^{(k)}\mid \alpha\in\Phi^+_{0}, r\in \Z, k\in \N\},
    \end{align} moreover,
    \begin{align}
        W_0\cdot&\{(x^+_{\alpha,r})^{(k)},(\frac{1}{2}X^+_{2,2r+1})^{(k)},(X^+_{1,2r+1})^{(k)}\mid \alpha\in \Phi^+_{0},r\in \Z, k\in \N\}=\\
        &\{(x^\pm_{\alpha,r})^{(k)},(\frac{1}{2}X^{\pm}_{i,2r+1})^{(k)}\mid \alpha\in \Phi^+_{0},i\in I, r\in \Z, k\in \N\},
    \end{align} then the claim follows observing that \begin{align}
        \tu^{+}\cap\{(x^\pm_{\alpha,r})^{(k)},(\frac{1}{2}X^{\pm}_{i,2r+1})^{(k)}\mid \alpha\in \Phi^+_{0},i\in I, r\in \Z, k\in \N\}=\\\{(x^+_{\alpha,r})^{(k)},(\frac{1}{2}X^{+}_{i,2r+1})^{(k)}\mid \alpha\in \Phi^+_{0},i\in I, r\in \Z, k\in \N\}.
    \end{align}   
\end{proof}
 $\uz^{0,\pm}$ is generated by the coefficients of $\hat h_1^\pm(u)$, $\bar h_1^\pm(u)$ and $\hat h_2^\pm(u)$ and $\uz^{\h}=\Z^{(bin)}[h_{i,0},c\mid i\in I]$
\begin{proposition}\label{menocartanpiu}
    The following identities hold in $\u[[u,v]]$:
\begin{align}
&\check h_1^+(u)\hat h_2^-(v)=\hat h_2^-(v)(1-uv)^{c}\check h_1^+(u)\label{eq:etadue},\\
&\hat h_1^+(u)\hat h_2^-(v)=\hat h_2^-(v)(1-uv)^{2c}\hat h_1^+(u)\label{eq:etaduedue},\\&\bar h_1^+(u)\hat h_2^-(v)=\hat h_2^-(v)(1-(uv)^2)^{c}\bar h_1^+(u)\label{eq:etaduetre}.
\end{align}
In particular $\uz^{0}=\uz^{0,-}\uz^{\h}\uz^{0,+}$ and $\uz^0$ is an integral form of $\tu^0$.
$\muz^{0}=\muz^{0,-}\muz^{\h}\muz^{0,+}$ and $\muz^0$ is an integral form of $\tu^0$.
\end{proposition}
\begin{proof}
    %Equation \eqref{eq:etaeta} follows from propositions \ref{heise} with $m=1$ and $l=\frac{1}{2}$.\\
    %Equations \eqref{eq:capcap}, \eqref{eq:barbar} and \eqref{eq:capbar},  follow from equation \eqref{eq:etaeta} and remark \ref{leacca}. 
    Equation %\eqref{eq:duedue} and 
    \eqref{eq:etadue} follows from \cite{DP},Propositions 2.11) %\ref{heise} 
    with $m=1$ and $l=0$, hence $\muz^{0}=\muz^{0,-}\muz^{\h}\muz^{0,+}$.
    Equations \eqref{eq:etaduedue} and \eqref{eq:etaduetre} follow form \eqref{eq:etadue} remembering that $\bar h_1^+(u)=\check h_1^+(u)\check h_1^+(-u)$ and $\hat h_1^+(u)=\check h_1^+(u)^2$,  hence $\uz^{0}=\uz^{0,-}\uz^{\h}\uz^{0,+}$
\end{proof}

\begin{lemma}
\label{cartantutta}
    $\buz^\pm\uz^\h=\uz^\h\buz^\pm$
\end{lemma}
\begin{proof}
From \cite{DP},Proposition 2.4)  with $m=a_{i,j}$ and we have that 
    \begin{align}
(x_{i,r}^+)^{(k)}\binom{h_{0,j}}{l}=\binom{h_{0,j}-a_{i,j}}{l}(x_{i,r}^+)^{(k)},
    \end{align} from \cite{DP},Proposition 2.4) and Equation \eqref{eq:formulaE}
    with $m=2$ by multiplying both side for $(\frac{1}{2})^k$ we have that: \begin{align}
        (\frac{1}{2}X_{1,2r+1}^+)^{(k)}
        \binom{h_{0,2}}{ l}=\binom{h_{0,2}-2k}{ l}
        (\frac{1}{2}X_{1,2r+1}^+)^{(k)}.
    \end{align} Hence we have that \begin{align}\label{eq:nonsopiucomechiamarle}
        \buz^+\uz^\h=\uz^\h\buz^+,
    \end{align} remarking that the $(\frac{1}{2}X_{1,2r+1}^+)^{(k)}$'s  and $(x_{i,r}^+)^{(k)}$s generate $\buz$, then by applying $\Omega$ to Relation \eqref{eq:nonsopiucomechiamarle} we get 
    \begin{align}
        \buz^-\uz^\h=\buz^\h\uz^-.
    \end{align}
\end{proof}

\begin{proposition}\label{commuzeropiupiu}
    The following identities hold  in $\bu[[u]]$
    \begin{align}
    &x_{1,0}^+\hat h_2^+(u)=\hat h_2^+(u)(1+uT^{-1})(x^+_{1,0}),\label{eq: commupiuuno}\\
    %&x_{2,0}^+\hat h_1^+(u)=\hat h_1^+(u)(1+uT^{-1})x^+_{2,0},\label{eq: commupiudueeta}\\
    &x_{2,0}^+  \check h_1^+(u)=\check h_1^+(u)(1+uT^{-1})(x_{2,0}^+)\label{eq: commupiunuova}\\
    %&x_{2,0}^+\hat h_1^+(u)=\hat h_1^+(u)(1+uT^{-1})x^+_{2,0},\label{eq: commupiudue}\\
    %&x^+_{2,0}\bar h_1^+(u^2)=\bar h_1^+(u^2)(1-u^2T^{-2})x^+_{2,0}\label{eq: commupiutre},\\
    %&x^+_{2,0}\mathring h_1^+(u)=\mathring h_1^+(u)(1-2uT^{-1})x^+_{2,0}\label{eq: commupiuquattro},\\ 
    &\frac{1}{2}X_{1,1}^+\hat h_2^+(u)=\hat h_2^+(u)(1+Tu^2)(\frac{1}{2}X_{1,1}^+).\label{eq: commupiuultima}
    \end{align}
    hence for all $k\in \N$
    \begin{align}
    %&\big(x_{1,0}^+\big)^{(k)}\eta_1^+(u)=\eta_1^+(u)\big((1-T^{-1}u)^{-1}(1-T^{-2}u^2)^{-3}x_{1,0}^+\big)^{(k)}\label{eq:kuno}\\
    &(x_{1,0}^+)^{(k)}\hat h_2^+(u)=\hat h_2^+(u)\big((1+uT^{-1})(x^+_{1,0})\big)^{(k)}\in \buz^{0,+}\buz^{+},\label{eq:kdue}\\
    &(x_{2,0}^+)^{(k)}  \check h_1^+(u)=\check h_1^+(u)((1+uT^{-1})(x_{2,0}^+))^{(k)}\in \muz^{0,+}\buz^{+}\label{eq:knuova}\\
    %&\big(x_{2,0}^+\big)^{(k)}\hat h_1^+(u)=\hat h_1^+(u)\big((1+uT^{-1})x^+_{2,0}\big)^{(k)},\label{eq:ktre}\\
    %&\big(x^+_{2,0}\big)^{(k)}\bar h_1^+(u^2)=\bar h_1^+(u^2)\big((1-u^2T^{-2})x^+_{2,0}\big)^{(k)},\label{eq:kquattro}\\
    %&\big(x^+_{2,0}\big)^{(k)}\mathring h_1^+(u)=\mathring h_1^+(u)\big((1-2uT^{-1})x^+_{2,0}\big)^{(k)}.\label{eq:kcinque}
    &(\frac{1}{2}X_{1,1}^+)^{(k)}\hat h_2^+(u)=\hat h_2^+(u)((1+Tu^2)(\frac{1}{2}X_{1,1}^+))^{(k)}\in \buz^{0,+}\buz^{+}.\label{eq:knuovanuova}
    \end{align}
    In particular $\buz^{0,+}\buz^\pm=\buz^\pm\buz^{0,+}$, $\buz^{0,-}\buz^\pm=\buz^\pm\buz^{0,-}$, $\muz^{0,+}\buz^\pm=\buz^\pm\muz^{0,+}$ and $\muz^{0,-}\buz^\pm=\buz^\pm\muz^{0,-}$, moreover are integral form of respectively $\tu^\pm\bu^{0,+}$ and  $\tu^\pm\bu^{0,-}$.
    $\muz^{0}\buz^\pm=\buz^\pm\muz^{0}$ and  $\buz^{0}\buz^\pm=\buz^\pm\buz^{0}$
    are integral form of $\tu^\pm\bu^{0}$.
\end{proposition}
\begin{proof}
%Equation \eqref{eq:commupieta} follows from proposition \ref{hh} with $m_1=1$, $m_2=3$ and $m_d=0$ if $d>2$.
Equations \eqref{eq: commupiuuno} and \eqref{eq: commupiunuova} and follow from \cite{DP},Proposition 2.14
%\ref{hh} 
with $m_1=-1$ and $m_d=0$ if $d>1$, Equation \eqref{eq: commupiuultima} follows \eqref{eq: commupiuuno} and \eqref{eq: commupiunuova}. %Equation \eqref{eq: commupiutre} follows from proposition \ref{hh} with $m_2=2$ and $m_d=0$ if $d\neq 2$.\\
%Equation \eqref{eq: commupiuquattro} follows from proposition \ref{hh} with $m_1=1$ and using $2u$ instead of $u$.
%Equations %\eqref{eq:kuno}, 
%\eqref{eq:kdue}, \eqref{eq:ktre} and \eqref{eq:kquattro} %and \eqref{eq:kcinque} 
%follow respectively from equations %\eqref{eq: commupieta}, 
%\eqref{eq: commupiuuno}, \eqref{eq: commupiudue}, \eqref{eq: commupiutre} %and \eqref{eq: commupiuquattro} 
Equations \eqref{eq:kdue}, \eqref{eq:knuova}  and \eqref{eq:knuovanuova} follow from \eqref{eq: commupiuuno} and \eqref{eq: commupiunuova}
since $\buz^+$ is $T$-stable and $T\vert_{\buz^{0,+}}=\text{id}$.
$\buz^+\buz^{0,+}=\buz^{0,+}\buz^+$ and $\buz^+\muz^{0,+}=\muz^{0,+}\buz^+$ follow directly, the others follow by applying $\Omega\circ \sigma$ and $\lambda_{-1}$. The last Relation follows from previous Relation and Lemma \ref{cartantutta}.
\end{proof}
%\begin{remark}    We know from Chapter \ref{theothers} that $\hat h^+_1(u), \bar h_1^+(u)\in \buz$, hence $\buz^0\subseteq \buz$.\end{remark}
We can now recollect the result of this subsection in the following Theorems:
\begin{theorem}\label{aquattrobuz}
$\uz$ it is an integral form of $\u$. More precisely:
    \begin{itemize}
        \item $\uz\cong\buz^-\otimes\buz^0\otimes\buz^+$,
        \item  $\buz^\pm=\buz\cap\u^\pm\neq\uz^\pm$;
        \item $\buz^0=\uz\cap\u^0$;
        \item $\buz^\h=\uz\cap\u^\h$;
        \item $\buz^{0,\pm}=\uz\cap\u^{0,\pm}$;
    \end{itemize}
\end{theorem}
\begin{proof}
From Theorem \ref{buzpiudue} it follows that  $\uz\supseteq\buz^+,\buz^-$, the other inclusion follows by the very Definitions of $\uz,\buz^+,\buz^-$. From Remark \ref{embeddings} and Corollary \ref{menocartanpiu} it follows that $\uz\supseteq\buz^{0,-},\buz^{0,\pm},\buz^{\h}$ and $\buz^{0}=\buz^{-}\otimes\buz^{0}\otimes\buz^{+}$. From Proposition \ref{commuzeropiupiu} it follows that $\uz\supseteq\buz^{-}\otimes\buz^{0}\otimes\buz^{\h}$. Since the generators of $\uz$ belongs to $\buz^{-}\otimes\buz^{0}\otimes\buz^{+}$ all the claims follow.
\end{proof}
\begin{theorem}
    $\muz\cong \buz^+\otimes\muz^0\otimes\buz^-\supsetneq\uz$ is an integral form of $\u$.
\end{theorem}
\begin{proof} The claim follows from Theorem \ref{aquattrobuz} and Proposition \ref{commuzeropiupiu}. 
\end{proof}

\subsection{$A_{2n}^{(2)}$, $n\ge2$}\label{ultimaspero}
We want now prove that $\uz(A_{2n}^{(2)})$ for all $n\ge2$ is an integral form of $\u$, more precisely we want to prove the following Theorem:
\begin{theorem}\label{main}
The $\Z$-subalgebra $\uz$ of $\u(A_{2n}^{(2)})$ generated by \begin{align}
    \{(x_{i,r}^+)^{(k)},(x_{i,r}^-)^{(k)}
    \mid i \in I, r \in \Z, k\in \N\}
\end{align} is an integral form of $\tu$.\\
More precisely
\begin{align}
\uz&\cong \buz^-\otimes \buz^0\otimes\buz^+\\
\buz^{0}&\cong \buz^{0,-}\otimes \buz^\h\otimes\buz^{0,+}\\
\end{align}
$\buz^\pm=\uz^\pm\cap\u$  is the $\Z$ linear span of the ordered monomials in \begin{align}
\{(x_{\alpha,r}^\pm)^{(k)}, (\frac{1}{2}X_{i,2r+1}^{\pm})^{(k)}\mid r\in \Z, \alpha\in \Phi_0^+,i\in I \},
    \end{align}
$\uz^{0,\pm}=%\Z^{(mix)}[h_{1,r}\mid \pm r>0]\otimes \Z[\hat h_{i,r}\mid \pm r>0, i=2,\dots, n]=
\uz^{0,\pm}\cap\u$,  
$\buz^{\h}%=\Z^{(bin)}[h_{i,r}, c\mid i\in I]
=\uz^{\h}\cap\u$.\\
\end{theorem}

We will prove Theorem \ref{main} by induction on $n$.
The claim for $n=2$ is the Subsection \ref{glue}. 
Since we should simultaneously work in this Section with different sets of indices, to emphasize the dependence on $n$ we will denote $I$ by $I_n$, namely $I_n=\{1,\dots, n\}$.
\begin{lemma}\label{mettiuno}
\begin{align}
\buz^{0}=\buz^{0,+}\otimes\buz^{\h}\otimes\buz^{0,-}\\
\muz^{0}=\muz^{0,+}\otimes\muz^{\h}\otimes\muz^{0,-}.\end{align} 
$\buz^{0}$ and $\muz^{0}$
are integral form of $\u^0$,
$\buz^{0,\pm}$ and $\muz^{0,\pm}$
are respectively an integral form of $\u^0$,
$\buz^{\h}$ and $\muz^{\h}$
are an integral form of $\u^\h$
and $\buz^{}\subseteq\uz$.
\end{lemma}
\begin{proof}
The claim follows from the embedding maps \eqref{eq:psi} and \eqref{eq:tildepsi} observing that $h_{i,r}$ and $h_{j,s}$ commute if $\vert i-j\vert\neq1$. 
\end{proof}

\begin{proposition}\label{mettidue}
 \begin{align}
     \buz^\pm \uz^{0}=\uz^{0} \buz^\pm
 \end{align}
\end{proposition}
\begin{proof} The proof is the same as Proposition \ref{cartantutta}.
\end{proof}

\begin{remark}
It follows from induction hypothesis that 
\begin{align}
\buz^\pm\subseteq\uz.  
\end{align} From Proposition \ref{mettidue} it follows that the claim of Theorem \ref{main} is equivalent to prove that  $\buz$ is an integral form of $\bu$ and a $\Z$ basis is given by the ordered monomials in the set 
\begin{align}
 \{(x_{\alpha,r}^\pm)^{(k)}, (\frac{1}{2}X_{i,2r+1}^{\pm})^{(k)}\mid r\in \Z, \alpha\in \Phi_0^+,i\in I \}.   
\end{align}Let us remark that we can restrict our consideration to only the $+$ case, the other will follow by applying $\Omega$.
\end{remark}
\begin{proposition}
\begin{align}
(x^+_{\alpha, r})^{(k)},(\frac{X^+_{i, 2r+1}}{2})^{(k)}\in\uz.   
\end{align} for any $\alpha\in \Phi_0^+ $, $i\in I$, $r\in\Z$ and $k\in \N$.
\end{proposition}
\begin{proof}
From the induction hypothesis, the  embedding map and \eqref{eq:tildepsi} an the $\tau_i$ stability of $\buz^+$ it follows that
\begin{align}
(\frac{X^+_{i, 2r+1}}{2})^{(k)}\in\uz.   
\end{align} for any $i\in I$, $r\in\Z$ and $k\in \N$, observing that \begin{align}
    (\frac{X^+_{n, 2r+1}}{2})^{(k)}=\tau_n((\frac{X^+_{n-1, 2r+1}}{2})^{(k)}).
\end{align}
Given $\alpha\in\Phi^+_0$, $\alpha=\sum_i a_i \alpha_i$, if $a_1=0$ or $a_n=0$, from the induction hypothesis and the embedding maps \eqref{eq:psi} and  \eqref{eq:tildepsi} it follows that 
\begin{align}
(x^+_{\alpha, r})^{(k)}\in\uz,   
\end{align} $r\in\Z$ and $k\in \N$.
Thus we can restrict to prove that:  
\begin{align}
(x^+_{\alpha, r})^{(k)}\in\uz   
\end{align} 
then $\forall k\in \N$, $r\in \Z$  and
\begin{align}
 \alpha\in\{\alpha_1+\dots+\alpha_n,2\alpha_1+\dots+2\alpha_j+2\alpha_{j+1}+\dots+\alpha_n\mid 1\le j<n\}.
\end{align}
Let us notice that 
\begin{align}
&\sigma_n(\alpha_1+\dots+\alpha_{n-1})=\alpha_1+\dots+\alpha_{n};\\
&\sigma_n(2\alpha_1+\dots+2\alpha_j+2\alpha_{j+1}+\dots+\alpha_{n-1})=2\alpha_1+\dots+2\alpha_j+2\alpha_{j+1}+\dots+\alpha_{n} \text{ if } j\neq n-1;\\
&\sigma_{n-1}\sigma_n(2\alpha_1+\dots+2\alpha_{n-2}+\alpha_{n-1})=\sigma_{n-1}(2\alpha_1+\dots+2\alpha_{n-2}+\alpha_{n-1}+\alpha_{n})=
2\alpha_1+\dots+2\alpha_{n-1}+\alpha_{n}.
\end{align}
thus 
the claim follows.
\end{proof}
\begin{theorem}
 \begin{align}
 (x^+_{\alpha, r})^{(k)}(x^+_{\beta, s})^{(l)}\in \uz^+,\label{eq:equaA}\\ 
 (x^+_{\alpha, r})^{(k)}(\frac{X^+_{i, 2s+1}}{2})^{(l)}\in \uz^+,\label{eq:equaB}\\  
 \end{align} 
 $\forall\alpha, \beta\in \Phi_0^+ $, $i\in I$, $r\in\Z$ and $k,l\in \N$.
\end{theorem}
\begin{proof}
Let $\alpha=\sum a_i \alpha_i$  and $\beta=\sum b_i \beta$.
be the decomposition of $\alpha$ and $\beta$ into simple roots.\\
Proof of Relation \eqref{eq:equaB}.
We will show that there exist $w\in W_0$, such that $w(\alpha)=\sum_{j=1}^{n-1}a_j\alpha_j$ and $w(2\sum_{j=1}^{i}\alpha_j)=2\sum_{j=1}^{l}\alpha_j$ with $l<n$, from which the claim follows from induction hypothesis.
\begin{itemize}
\item If $a_n=0$ and $i<n$ there is nothing to prove.\\
\item 
If $i=n$ Relation \eqref{eq:equaB} follows observing that $X^+_{i, 2s+1}$ is central in $\uz^+$.
\item 
If $a_n=1$ and $i<n$, let us notice that $\alpha+2\sum_{j=1}^i\alpha_j \in \Phi_0^{+}$
if  only if $\alpha=\alpha_{i+1}+\dots+\alpha_n$,thus
the claim follows observing that 
\begin{align}
&\sigma_n(2\sum_{j=1}^i\alpha_j)=2\sum_{j=1}^i\alpha_j  &\text{ if } i\neq n-1; \\
&\sigma_n(\alpha_{i+1}+\dots+\alpha_n) =\alpha_{i+1}+\dots+\alpha_{n-1} &\text{ if } i\neq n-1;\\
&\sigma_{n}\sigma_{n-1}\sigma_{n-2}(2\sum_{j=1}^{n-1}\alpha_j)=\sigma_{n}\sigma_{n-1}(2\sum_{j=1}^{n-2}\alpha_j)=\sigma_{n}(2\sum_{j=1}^{n-1}\alpha_j)=2\sum_{j=1}^{n}\alpha_j
; \\
&\sigma_{n}\sigma_{n-1}\sigma_{n-2}(\alpha_n)=\sigma_{n}(
\alpha_{n-1}+\alpha_n)=\alpha_{n-1}.
\end{align}
\end{itemize}
Proof of Relation \eqref{eq:equaA}.\\
We will show that there exist $w\in W_0$, such that $w(\alpha)=\sum_{j=1}^{l}a_j'\alpha_j$ and $w(\beta)=\sum_{j=1}^{k}b_j'\alpha_j$ with $l,k<n$, from which the claim follows from induction hypothesis.
If $a_1+b_1>2$ or $a_n+b_n>2$ then $\alpha+\beta+k\delta\not\in\Phi$ for any $k\in\Z$, the cases to consider are therefore the following.
\begin{itemize}
\item If $a_1=b_1=0$ or $a_n=b_n=0$ there is nothing to prove.
\item If 
$a_1=1$, $b_1=1$
$a_n=1$, $b_n=1$,  
that is $\alpha=\beta=\alpha_1+\dots+\alpha_n$ 
then $\sigma_n(\alpha)=\sigma_n(\beta)=\alpha_1+\dots+\alpha_{n-1}$.
\item If $a_1=2$, $b_1=0$
$a_n=1$, $b_n=1$, that is 
$\alpha=2\alpha_1+\dots+2\alpha_j+\alpha_{j+1}+\dots+\alpha_n$ and $\beta=\alpha_r+\dots+\alpha_n$
thus 
$\alpha+\beta\in \Phi^+_0$ if and only 
$r=j+1$, that is 
$\alpha=2\alpha_1+\dots+2\alpha_j+\alpha_{j+1}+\dots+\alpha_n$ and $\beta=\alpha_{j+1}+\dots+\alpha_n$

\begin{itemize}
    \item 
If $j+1=n-1$
that is 
$\alpha=2\alpha_1+\dots+2\alpha_{n-1}+\alpha_n$ 
then $\sigma_{n}
\sigma_{n-1}(\alpha)=\sigma_n(2\alpha_1+\dots+2\alpha_{n-2}+\alpha_{n-1}+\alpha_n)=2\alpha_1+\dots+2\alpha_{n-2}+\alpha_{n-1}$
and 
$\sigma_{n}
\sigma_{n-1}(\beta)=\sigma_n(\alpha_{n-1}+\alpha_n)=\alpha_{n-1}$.
\item 
if $j+1<n-1$
then 
$\sigma_{n}(\alpha)=\sigma_n(2\alpha_1+\dots+2\alpha_{j+1}+\alpha_{j}+\dots+\alpha_n)=2\alpha_1+\dots+2\alpha_{j+1}+\alpha_{j}+\dots+\alpha_{n-1}$ and
$\sigma_{n}
(\beta)=\sigma_n(\alpha_{j+1}+\dots+\alpha_n)=\alpha_{j+1}+\dots+\alpha_{n-1}$.
\end{itemize}

\item If 
$a_1=1$, $b_1=0$
$a_n=1$, $b_n=1$, that is 
$\alpha=\alpha_1+\dots+\alpha_n$ and
$\beta=\alpha_r+\dots+\alpha_n$
then
$\alpha+\beta+k\delta\not \in \Phi_0$  $\forall k\in \Z$.
\item If 
$a_1=2$, $b_1=0$
$a_n=1$, $b_n=0$ , that is $\alpha=2\alpha_1+\dots+2\alpha_j+\alpha_{j+1}+\dots+\alpha_n$ and $\beta=\alpha_r+\dots+\alpha_s$, then 
$\alpha+\beta\in \Phi^+_0$ if and only 
$r=j+1$.
Let us assume then $r=j+1$, that is 
$\alpha=2\alpha_1+\dots+2\alpha_j+\alpha_{j+1}+\dots+\alpha_n$ and
$\beta=\alpha_{j+1}+\dots+\alpha_s$.
%\begin{itemize}\item RIVEDI    \end{itemize}

\item If 
$a_1=2$, $b_1=0$
$a_n=0$, $b_n=1$ , that is $\alpha=2\alpha_1+\dots+2\alpha_j+\alpha_{j+1}+\dots+\alpha_k$ and $\beta=\alpha_r+\dots+\alpha_n$, then $\alpha+\beta\in \Phi^+_0$ if and only $r=j+1$, that is 
$\alpha=2\alpha_1+\dots+2\alpha_j+\alpha_{j+1}+\dots+\alpha_k$ and $\beta=\alpha_{j+1}+\dots+\alpha_n$
%\begin{itemize}    \item RIVEDI\end{itemize}

\item If 
$a_1=1$, $b_1=0$
$a_n=0$, $b_n=1$, that is
$\alpha=\alpha_1+\dots+\alpha_j$ and
$\beta=\alpha_r+\dots+\alpha_n$, then
$\alpha+\beta\in \Phi^+_0$ if and only
$r=j+1$, that is
$\alpha=\alpha_1+\dots+\alpha_j$ and
$\beta=\alpha_{j+1}+\dots+\alpha_n$.
\begin{itemize}
    \item If $j+1=n$ then
    $\sigma_{n}\sigma_{n-1}(\alpha)=\alpha_1+\dots+\alpha_{n-2}$ and
$\sigma_{n}\sigma_{n-1}(\beta)=\alpha_{n-1}$.
\item  If $j+1\neq n$ then
    $\sigma_{n}(\alpha)=\alpha$ and
$\sigma_{n}=\alpha_{j+1}+\dots+\alpha_{n-1}$.
\end{itemize}
\end{itemize}
\end{proof}

\begin{theorem}\label{theoremfinalepiu}
    $\buz^+$ and $\buz^-$ are integral form of $\tu^+$ and $\tu^-$, more precisely %\begin{align} \buz^\pm=\Z^{(div)}[x^{\pm}_{\alpha       ,r}, \frac{1}{2}X^{\pm}_{i        ,2r+1}\mid \alpha\in \Phi_{0}^+, i \in I, r\in \Z],    \end{align} 
      a basis of $\buz^\pm$ is given by the divided powers of the elements of the set 
     a $\Z$-basis $B^{\pm}$
     $\{x^{\pm}_{\alpha
        ,r}, \frac{1}{2}X^{\pm}_{i
        ,2r+1}, \alpha\in \Phi_{0}^+, i \in I, r\in \Z\}$.
\end{theorem}

%\printbibliography

\end{document}